\documentclass[3p]{elsarticle}
\usepackage{fullpage}

\usepackage{amsmath,amsfonts,amsthm,amssymb}
\usepackage{natbib}

\usepackage{graphicx}                          
\usepackage{setspace}  
\usepackage{caption}
\usepackage{listings}
\usepackage{tabularx}
\usepackage{xcolor}
\usepackage{indentfirst}
\usepackage{subcaption}

\usepackage{multirow}
\usepackage{overpic}
\usepackage{diagbox}

\usepackage{verbatim}

\usepackage{enumerate}

\usepackage{tikz}

\usepackage[title]{appendix}

\usepackage{bm}

\usepackage[font=footnotesize,labelfont=bf]{caption}

\usepackage{nicefrac}

\makeatletter
\@addtoreset{equation}{section}
\makeatother

\newtheorem{theorem}{Theorem}[section]

\newtheorem{remark}{Remark}[section]
 
\newtheorem{definition}{Definition}

\newcommand\dd{\mathrm{d}}
\newcommand\pp{\partial}

\newcommand\x{\bm{x}}
\newcommand\z{\bm{z}}
\newcommand\uvec{\mathbf{u}}
\newcommand\X{\mathbf{X}}

\newcommand\n{{\bf n}}

\makeatletter
\def\widebreve{\mathpalette\wide@breve}
\def\wide@breve#1#2{\sbox\z@{$#1#2$}%
     \mathop{\vbox{\m@th\ialign{##\crcr
\kern0.08em\brevefill#1{0.8\wd\z@}\crcr\noalign{\nointerlineskip}%
                    $\hss#1#2\hss$\crcr}}}\limits}
\def\brevefill#1#2{$\m@th\sbox\tw@{$#1($}%
  \hss\resizebox{#2}{\wd\tw@}{\rotatebox[origin=c]{90}{\upshape(}}\hss$}
\makeatletter

	\newcommand\be {\begin{equation}}
	\newcommand\ee {\end{equation}}

	\newcommand\dt {{\Delta t}}

\usepackage{mhchem}

\renewcommand{\dot}[1] {\overset{\,_{\mbox{\huge .}}}{#1}}

\numberwithin{equation}{section}

\begin{document}

\title{A structure-preserving, operator splitting scheme for reaction-diffusion equations with detailed balance}

\author[iit]{Chun Liu}
\ead{cliu124@iit.edu}

\author[umnd]{Cheng Wang}
\ead{cwang1@umassd.edu}

\author[iit]{Yiwei Wang\corref{cor1}}
\cortext[cor1]{Corresponding author}
\ead{ywang487@iit.edu}

\address[iit]{Department of Applied Mathematics, Illinois Institute of Technology, Chicago, IL 60616, USA}
\address[umnd]{Department of Mathematics, University of Massachusetts, North Dartmouth, MA 02747}

\begin{abstract}
In this paper, we propose and analyze a positivity-preserving, energy stable numerical scheme for certain type reaction-diffusion systems involving the Law of Mass Action with the detailed balance condition. The numerical scheme is constructed based on a recently developed energetic variational formulation, in which the reaction part is reformulated in terms of reaction trajectories. 
The fact that both the reaction and the diffusion parts dissipate the same free energy opens a path of an energy stable, operator splitting scheme for these systems. At the reaction stage, we solve equations of reaction trajectories by treating all the logarithmic terms in the reformulated form implicitly due to their convex nature. The positivity-preserving property and unique solvability can be theoretically proved, based on the singular behavior of the logarithmic function around the limiting value. Moreover, the energy stability of this scheme at the reaction stage can be proved by a careful convexity analysis. Similar techniques are used to establish the positivity-preserving property and energy stability for the standard semi-implicit solver at the diffusion stage. As a result, a combination of these two stages leads to a positivity-preserving and energy stable numerical scheme for the original reaction-diffusion system. To our best knowledge, it is the first time to report an energy-dissipation-law-based operator splitting scheme to a nonlinear PDE with variational structures. Several numerical examples are presented to demonstrate the robustness of the proposed operator splitting scheme.
\end{abstract}


\maketitle




\section{Introduction}

Reaction-diffusion type equations play an important role in modeling many physical and biological processes, such as pattern formation \cite{hao2020spatial, kondo2010reaction, pearson1993complex}, tumor growth \cite{hawkins2012numerical, liu2018accurate, perthame2014hele} and molecular motors \cite{chipot2003variational, julicher1997modeling}. Mathematically, a reaction-diffusion system often takes a form of 
\begin{equation} \label{equation: reaction-diffusion} 
\pp_t \bm{c} = \nabla \cdot (D(\bm{c}, \x) \nabla \bm{c}) + \bm{r}(\bm{c}),
\end{equation}
subject to a suitable boundary condition, where $\bm{c}(\x, t)$ represents the number density of each species in this system, $\bm{r}(\bm{c})$ accounts for all local reactions, and $D(\bm{c}, \x)$ is the  diffusion coefficient that depends on $\bm{c}$ and location $\bm{\x}$. There are substantial interests in studying reaction-diffusion systems both analytically and numerically, see \cite{bertolazzi1996positive, desvillettes2006, desvillettes2017trend, formaggia2011positivity, glitzky2013gradient, haskovec2018decay, huang2019positivity, lo2012robust, mielke2011gradient, mielke2013thermomechanical, zhao2011operator} for instance. 

The numerical computation of reaction-diffusion systems is often very challenging, due to the nonlinearity and stiffness brought by the reaction term \cite{nie2006efficient}. In addition, some standard numerical schemes may not be able to preserve the positivity and conservation properties of species densities, which is a natural requirement for any numerical approximation \cite{formaggia2011positivity}.
Although many numerical schemes have been developed to solve reaction kinetics and reaction-diffusion systems \cite{bertolazzi1996positive, formaggia2011positivity, huang2019positivity, zhao2011operator}, a theoretical justification of these properties has been very limited.
Recently, there have been growing interests in developing structure-preserving numerical schemes for gradient-flow type PDEs with energy-dissipation laws.
Examples include the convex splitting idea~\cite{wise2009energy}, invariant energy quadratization (IEQ) \cite{zhao2016decoupled}, and scalar auxiliary variable (SAV) scheme~\cite{shen2019new}. Some theoretical analysis of positivity-preserving property and energy stability have been reported for these numerical schemes for certain systems with singular energy potential, such as Poisson-Nernst-Planck (PNP) system  \cite{shen2020PNP}, and the Cahn-Hilliard equation with Flory-Huggins energy~\cite{chen19b, dong20b, dong19a, dong20a}. However, for a general reaction-diffusion system~\eqref{equation: reaction-diffusion}, a variational structure may not exist, so that these structure-preserving ideas may be not directly applicable. Fortunately, for certain type reaction-diffusion systems, in which the reaction part satisfies the law of mass action with \emph{detailed balance} conditions, it has been discovered that, these systems admit Lyapunov functions and possess variational structures \cite{anderson2015lyapunov, desvillettes2006, desvillettes2017trend, glitzky2013gradient, haskovec2018decay, liero2013gradient, mielke2011gradient, mielke2013thermomechanical, shear1967analog, wang2020field}, which opens a new door to develop a structure-preserving numerical scheme for the systems.

In this paper, we propose and analyze a positivity-preserving and energy-stable operator splitting scheme for reaction-diffusion systems with variational structures, based on an energetic variational formulation proposed in \cite{wang2020field}.
In this formulation,  the reaction part is reformulated in terms of reaction trajectories ${\bm R}$, and the reaction and diffusion parts impose different dissipation mechanisms for the same physical energy. As a result, it opens a path of an operator splitting numerical scheme that preserves the energy dissipation property of the whole system. In more details, since the same physical energy dissipates at both the reaction and the diffusion stages, the combined operator splitting scheme yields a dissipated energy at each time step, with a careful transformation between two splitting stages. 
 The idea of operator splitting schemes has been applied to many nonlinear PDEs \cite{Bao2002, Thalhammer2012}, including reaction-diffusion systems \cite{descombes2001convergence, zhao2011operator} for a long time. However, most existing works didn't use the variational structure, or have different physical energy in different computational stages so that a combined energy stability estimate becomes very challenging at a theoretical level. In our proposed approach, since the physical free energy is in the form of logarithmic functions of the concentration ${\bm c}$, which is a linear function of reaction trajectories ${\bm R}$, the positivity-preserving analysis of the numerical scheme at both stages can be established. Similar to the analysis in a recent article~\cite{chen19b} for the Flory-Huggins Cahn-Hilliard flow, an implicit treatment of the nonlinear singular logarithmic term is crucial to theoretically justify its positivity-preserving property. A more careful analysis reveals that, the convex and the singular natures of the implicit nonlinear parts prevents the numerical solutions approach the singular limiting values, so that the positivity-preserving property is available for all the species density variables. 
To fix the idea, we only consider a reaction-diffusion systems taking the form
 \begin{equation}
  \pp_t c_i = \nabla \cdot (D_i(c_i) \nabla c_i) + r_i ({\bm c}), \quad i = 1, \ldots N,
 \end{equation}
where $D_i(c_i)$ only depend on $c_i$, so that the diffusion part is fully decoupled. 
The presented numerical approach can be applied to a more complicate system that couples chemical reaction with cross-diffusion, or Cahn-Hilliard or PNP type diffusions \cite{giga2017variational}.
 
The rest of this paper is organized as follows. In Section 2, we introduce the energetic variational approach to reaction-diffusion systems with detailed balance, which forms a basis of the numerical scheme design. The operator splitting scheme is presented in Section 3, along with a detailed analysis on the positivity-preserving property and unconditional energy stability. Finally, some numerical results are shown in Section 4 to demonstrate the robustness of the proposed operator splitting scheme.

\section{Energetic variational approach}

Originated from seminal works of Rayleigh \cite{rayleigh1873note} and Onsager \cite{onsager1931reciprocal,onsager1931reciprocal2}, the Energetic Variational Approach (EnVarA) provides a paradigm to derive the dynamics of a complicated system from a prescribed \emph{energy-dissipation law}, through two distinct variational processes: the Least Action Principle (LAP) and the Maximum Dissipation Principle (MDP)~\cite{eisenberg2010energy, giga2017variational, liu2009introduction}. These approaches have been successfully applied to build up many mathematical models \cite{giga2017variational}, including systems with chemical reaction \cite{wang2020field}. In the meantime, the EnVarA also provides a guideline of designing structure-preserving numerical schemes for systems with variational structures \cite{liu2019lagrangian,liu2020variational}.

The starting point of the EnVarA is the first and second laws of thermodynamics, which leads to an energy-dissipation law
\begin{equation}
\frac{\dd}{\dd t} E^{\text{total}}(t) = - \triangle(t) \leq 0,
\end{equation}
for an isothermal closed system \cite{ericksen1992introduction, giga2017variational}. Here $E^{\text{total}}$ is the total energy, which is the sum of the Helmholtz free energy $\mathcal{F}$ and the kinetic energy $\mathcal{K}$; $\triangle(t)$ stands for the energy dissipation rate. 
The Least Action Principle (LAP) states that, the dynamics of the Hamiltonian part of the system is determined by a critical point of the action functional $\mathcal{A}(\x) = \int_0^T \mathcal{K} - \mathcal{F} \dd t$ with respect to $\x$ (the trajectory in Lagrangian coordinates, if applicable) \cite{arnol2013mathematical, giga2017variational}, which derives the conservative force
\begin{equation}
  \delta \mathcal{A} =  \int_{0}^T \int_{\Omega(t)} (f_{\text{inertial}} - f_{\text{conv}})\cdot \delta \x  \dd \x \dd t.
\end{equation}
On the other hand, according to Onsager \cite{onsager1931reciprocal, onsager1931reciprocal2}, the dissipative force in a dissipative system ($\triangle(t) \geq 0$) can be obtained by minimizing the Onsager dissipation functional $\mathcal{D} = \frac{1}{2} \triangle(t)$ with respect to the ``rate'' $\x_t$ in the linear response regime, i.e.,
\begin{equation}
\delta \mathcal{D}  = \int_{\Omega(t)} f_{\text{diss}} \cdot \delta \x_t~ \dd \x.
\end{equation}
This principle is know as Onsager's Maximum dissipation principle \cite{giga2017variational}.
Hence, the force balance condition (Newton's second law, in which the inertial force plays a role of $ma$) results in
\begin{equation}\label{FB}
\frac{\delta \mathcal{A}}{\delta \x} = \frac{\delta \mathcal{D}}{\delta \x_t},
\end{equation}
which defines the dynamics of the system.

In the framework of the EnVarA, a generalized diffusion for a conserved quantity $\rho$ can be derived from the
energy-dissipation law 
\begin{equation}\label{Diffusion}
\frac{\dd}{\dd t} \mathcal{F}[\rho] = - \int_{\Omega} \xi(\rho) |\uvec|^2 \dd \x,
\end{equation}
along with the kinematics
\begin{equation}\label{kinematic}
\rho_t + \nabla \cdot (\rho \uvec) = 0.
\end{equation}
Here $\mathcal{F}[\rho]$ is the free energy, $\uvec$ is the macroscopic velocity, and $\xi(\rho) > 0$ is the friction coefficient. The kinematics (\ref{kinematic}) is the consequence of the mass conservation of $\rho$. Through this paper, we assume that 
$$ 
  \mathcal{F}[\rho] = \int_{\Omega} \omega(\rho) \dd \x, 
$$
where $\omega(\rho)$ is the free energy density.

To derive the equation of $\uvec$ by using the EnVarA, one should reformulate the energy-dissipation law (\ref{Diffusion}) in Lagrangian coordinates. Let $\x(\X, t)$ be the flow map associated with the velocity field $u(\x(\X, t), t)$, i.e, 
$$ 
\frac{\dd}{\dd t}\x (\X, t) = \uvec(\x(\X, t), t) , 
$$ 
where $\X$ are Lagrangian coordinates. In Lagrangian coordinates, the conservation of mass indicates that
\begin{equation*}
\rho(\x(\X, t), t) = \frac{\rho_0(\X)}{\det {\sf F}}, 
\end{equation*}
where ${\sf F} = \nabla_{\X} \x(\X, t)$ is the deformation tensor associated with the flow map $\x(\X, t)$. 
Thus, the energy-dissipation law (\ref{Diffusion}) can be reformulated as
\begin{equation}
\frac{\dd}{\dd t} \int_{\Omega_0} \omega \left( \frac{\rho_0}{\det {\sf F}}  \right)   \dd \X = - \int \xi \left( \frac{\rho_0}{\det {\sf F}} \right) |\x_t|^2 \det {\sf F} \dd \X.
\end{equation}
in Lagrangian coordinates.
Performing the LAP, i.e., computing a variation of $\mathcal{A} = \int_{0}^{T} - \mathcal{F} \dd t $ with respect to $\x(\X, t)$, leads to 
\begin{equation}\label{LAP_Lag}
  \begin{aligned}
    & \delta \mathcal{A} = - \delta \int_{0}^T  \int_{\Omega_0} \omega(\rho_0(X)/ \det {\sf F}) \det {\sf F} \dd \X \\
    & = - \int_{0}^T \int_{\Omega_0} \left( - \omega_{\rho} \left( \frac{\rho_0(X)}{\det {\sf F}} \right) \frac{\rho_0(X)}{\det {\sf F}} + \omega\left(\frac{\rho_0(X)}{\det {\sf F}}\right) \right) ({\sf F}^{-\rm{T}} : \nabla_{\X} \delta \x)\det {\sf F}  \dd \X, \\
      \end{aligned}
\end{equation}
where $\widetilde{\delta \x}(\x(\X, t), t) = \delta \x(\X, t)$ is a test function satisfying $\widetilde{\delta \x} \cdot \n = 0$ and $\n$ is the outer normal of $\Omega$. Pushing forward to the Eulerian coordinate and applying the integration by parts, we have
\begin{equation}
  \begin{aligned}
\delta \mathcal{A} & = - \int_{0}^T \int_{\omega} ( - \omega_\rho \rho + \omega) \nabla \cdot (\delta \x) \dd \x 
 = \int_{0}^T \int_{\Omega} - \nabla \cdot (\omega_\rho \rho - \omega) \cdot \delta \x  \dd \x.
  \end{aligned} 
\end{equation}
Hence, 
$$
  \frac{\delta \mathcal{A}}{\delta \x} = - \nabla \cdot (\omega_\rho \rho - \omega) = - \nabla \mu, 
  \quad \mu = \omega_{\rho} = \frac{\delta \mathcal{F}}{\delta \rho},
$$ 
where $\mu$ is known as the chemical potential. In the meantime, the MDP yields $\dfrac{\delta \mathcal{D}}{\delta \x_t} = \xi(\rho) |\x_t|$. So the force balance (\ref{FB}) leads to the equation for $\uvec$, given by
\begin{equation}\label{eq_u}
\xi(\rho) \uvec = - \rho \nabla \mu.
\end{equation}

\begin{remark}
  Formally, the equation of $\uvec$ can be derived by using the fact that
  \begin{equation}
    \begin{aligned}
  \frac{\dd}{\dd t} \int \rho \ln \rho \dd \x & = \int (\ln \rho + 1) \rho_t \dd \x = \int (\ln \rho + 1) (- \nabla \cdot (\rho \uvec)) \dd \x \\
  & = \int \nabla  (\ln \rho + 1) \cdot (\rho \uvec) \dd \x . \\
    \end{aligned}
  \end{equation}
Therefore, the energy-dissipation law (\ref{Diffusion}) indicates that $\xi(\rho) \uvec = - \rho \nabla \mu$, which corresponds to the principle of virtual work.
\end{remark}

A combination of (\ref{eq_u}) with (\ref{kinematic}) leads to a diffusion equation
\begin{equation}\label{diffusion1}
\rho_t = \nabla \cdot \left (  \frac{\rho^2}{\xi(\rho)}  \nabla \mu(\rho) \right),  
\end{equation}
which can be rewritten as
\begin{equation}
\rho_t = \nabla \cdot \left(  D(\rho)  \nabla \rho\right),
\end{equation}
where $D(\rho) = \rho^2 \mu'(\rho) / \xi(\rho)$ is the nonlinear diffusion coefficients.
We refer the readers to \cite{giga2017variational, liu2019lagrangian} for more detailed descriptions. 

The energy-dissipation law (\ref{Diffusion}) has a natural $L^2$-gradient flow structure in the space of diffeomorphism, which provides a basis to study a structure preserving Lagrangian method for generalized diffusions \cite{carrillo2018lagrangian, junge2017fully, liu2019lagrangian}. To develop a energy stable Eulerian method, it is more convenient to formulate the energy dissipation law as
  \begin{equation}
    \frac{\dd}{\dd t} \int \omega(\rho) \dd \x = - \int \mathcal{M}(\rho) |\nabla \mu|^2 \dd \x,
  \end{equation}
 where $\mathcal{M}(\rho) = \dfrac{\rho^2}{\xi(\rho)}$ is called the mobility function, and $\mu$ is the chemical potential.

\subsection{EnVarA to reaction kinetics}


In general, a chemical reaction system containing $N$ species $\{ X_1, X_2, \ldots X_N \}$ and $M$ reactions can be represented by 
\begin{equation}
\ce{ $\alpha_{1}^{l} X_1 + \alpha_{2}^{l}X_2 + \ldots \alpha_{N}^{l} X_N$ <=> $\beta_{1}^{l} X_1 + \beta_{2}^{l}X_2 + \ldots \beta_{N}^{l} X_N$}, \quad l = 1, \ldots, M.
\end{equation}
Here we assume that all the reactions are reversible \cite{ge2017mathematical}. From a modeling perspective, an irreversible chemical reaction can be viewed as a singular limit of a reversible one \cite{gorban2013thermodynamics}.
In a spatially homogeneous case, the macroscopic kinetics of this system is described by a system of ordinary differential equations
\begin{equation}\label{CR_ODE}
\frac{\dd \bm{c}}{\dd t}  = \bm{\sigma} \bm{r}(\bm{c}(t), t),  
\end{equation}
where $\bm{c} = (c_1, c_2, \ldots, c_N)^{\rm T} \in \mathbb{R}^N_{+}$ denotes the concentrations of all species, $\bm{r} = (r_1, r_2 \ldots, r_M)^{\rm T} \in \mathbb{R}^M$ represents the corresponding reaction rates,  and $\bm{\sigma} \in \mathbb{R}^{N \times M}$ is the stoichiometric matrix, with the $(i, l)$ element given by $\sigma_{i}^l = \beta^l_i - \alpha^l_i$. 
The system (\ref{CR_ODE}) admits  $N - rank ({\bm \sigma})$ linearly independent conserved quantities (invariants), due to the fact that 
\begin{equation}\label{conserv}
\frac{\dd}{\dd t} (\bm{e} \cdot c) = \bm{e} \cdot \bm{\sigma} \bm{r}(\bm{c}(t), t) = 0,
\end{equation}
for $\bm{e} \in Ker(\bm{\sigma}^{\rm T})$. 
One challenge in designing an efficient numerical method for reaction kinetics system is to preserve the conservation properties (\ref{conserv}), as well as the positivity of ${\bm c}$, after numerical discretization is taken into consideration \cite{formaggia2011positivity, sandu2001positive}.

In classical chemical kinetics, it is often assumed that the reaction rate is directly proportional to the product of the concentrations of the reactants, known as the \emph{Law of Mass Action}, that is
\begin{equation}\label{LMA}
r_l (\bm{c}) = k_{l}^+ {\bm c}^{{\bm \alpha}^l} -  k_{l}^- {\bm c}^{{\bm \beta}^l} = 0,
\end{equation}
where $k_{l}^+$ and $k_{l}^-$ are the forward and backward reaction constants for the $l$-th reaction, and
\begin{equation}
{\bm c}^{{\bm \alpha}^l}  = \prod_{i=1}^N c_i^{\alpha_i^l}, \quad {\bm c}^{{\bm \beta}^l}  = \prod_{i=1}^N c_i^{\beta_i^l} . 
\end{equation}
It has been discovered for a long time that there exists a Lyapunov function for the reaction kinetics (\ref{CR_ODE}) with law of mass action (\ref{LMA}), if the system satisfies the \emph{detailed balance} condition \cite{desvillettes2017trend, mielke2011gradient, shear1967analog}, which is defined as follows \cite{desvillettes2017trend}. 
\begin{definition}
  An equilibrium point $\bm{c}_{\infty} \in \mathbb{R}^N_{+}$ is said to be a
  \emph{detailed balance} equilibrium if and only if any forward reaction is balanced with its corresponding backward reaction at the equilibrium, i.e.,
  \begin{equation}
    k_{l+} {\bm c}_{\infty}^{{\bm \alpha}^l} = k_{l-} {\bm c}_{\infty}^{{\bm \beta}^l} ,  
    \quad l  = 1, \ldots M . 
  \end{equation}
A chemical reaction network is called detailed balance if it possesses a detailed balance equation.
\end{definition}
For a detailed balanced reaction system, the Lyapunov function of the system is given by
\begin{equation}\label{free_energy_c}
\mathcal{F}[c_i] = \sum_{i=1}^N c_i \left( \ln \left( \frac{c_i}{c_i^{\infty}}  \right) - 1 \right),
\end{equation}
where ${\bm c}^{\infty}$ is a detailed balance equilibrium.
The Lyapunov function can be understood through the thermodynamics argument \cite{mielke2017non}
\begin{equation}\label{free_energy_U}
\mathcal{F}[c_i] = \sum_{i=1}^N \left( c_i (\ln c_i - 1) + c_i U_i \right),
\end{equation}
where the first part corresponds to the energy, and $U_i$ is the internal energy per mole associated with each species. The choice of $U_i$ determines the equilibrium of the system $c_i^{\infty}$, i.e.,
\begin{equation}\label{U_c_eq}
\sum_{i=1}^N \sigma_i^l (\ln c_i^{\infty} + U_i) = 0, \quad l = 1, \ldots M . 
\end{equation}
\begin{remark}
It is worth mentioning that the form of free energy (\ref{free_energy_U}) is also valid for a reaction network without detailed balance, as if $U_i$ given. However, for arbitrary given $U_i$, the system may not admit a chemical equilibrium \cite{ederer2007thermodynamically}.
Under the detailed balance condition, one can simply specify $U_i = - \ln c_i^{\infty}$ \cite{mielke2011gradient}. 
\end{remark}


Like other variational principles, the classical energetic variational approach
cannot be applied to chemical reaction system directly. During the last couple of decades, many papers try to build an Onsager type variational theory for chemical reaction systems \cite{bataille1978nonequilibrium, biot1977variational, beris1994thermodynamics,mielke2011gradient, oster1974chemical}.
Recently, by using the concept of reaction trajectory ${\bm R}$, originally introduced by De Donder \cite{de1936thermodynamic}, an energetic variational formula of reactions kinetics has been proposed in~\cite{wang2020field}. 
The reaction trajectories ${\bm R} \in \mathbb{R}^M$, which is an analogy to the flow map $\x(\X, t)$ in a mechanical system. Roughly speaking, the $l-$th component of ${\bm R}(t)$ accounts for the number of the $l$-th reaction happens in the forward direction by time $t$.
Given an initial condition ${\bm c}(0) \in \mathbb{R}^N_{+}$, the value of ${\bm c(t)}$ can be expressed in terms of ${\bm R}$ as
\begin{equation}\label{c_R_1}
      \bm{c}(t) = \bm{c}(0) + \bm{\sigma} {\bm R}(t),
\end{equation}
where ${\bm R}(t) \in \mathbb{R}^M$ represents $M$ reaction trajectories in the system. 
 Equation (\ref{c_R_1}) can be viewed as the kinematic of a reaction network, which embodies the conservation property (\ref{conserv}). The positivity of ${\bm c}$ requires that 
$$
  \bm{\sigma} {\bm R}(t) + \bm{c}(0) > 0,  
$$
which is a constraint on ${\bm R}$.

The reaction trajectories was originally introduced by De Donder as a new state variable for a chemical reaction system \cite{de1927affinite, kondepudi2014modern}, and has been widely used in different theoretical approaches to study chemical reactions \cite{anderson2015lyapunov, oster1974chemical}.
The reaction rate ${\bm r}$ can be defined as $\dot{\bm R}$. We abbreviate $\dot{\bm R} = \frac{\dd}{\dd t} {\bm R}$ throughout this paper. The reaction rate ${\bm r}$ or $\dot{\bm R}$ can be viewed as the reaction velocity \cite{kondepudi2014modern}.
In the framework of the energetic variational approach, we can describe the reaction kinetic of the system by imposing the energy-dissipation law in terms of ${\bm R}$ and $\dot{\bm R}$, that is
\begin{equation}\label{ED_R}
\frac{\dd}{\dd t} \mathcal{F}[{\bm R}] = - \mathcal{D}_{\rm chem}[{\bm R}, \dot{\bm R}],
\end{equation}
where $\mathcal{D}_{\rm chem}[{\bm R}, \dot{\bm R}]$ is the rate of energy-dissipation (entropy production) due to the chemical reaction procedure. 

Unlike the mechanical part, the linear response assumption may not be valid for a chemical system until the last stage of chemical reaction \cite{de2013non}. Therefore, the energy-dissipation rate $\mathcal{D}_{\rm chem}$ may not be quadratic in terms of $\dot{\bm R}$~\cite{beris1994thermodynamics, de2013non}. To  deal with the nonlinear dissipation, we have to extend the classical energetic variational approach. 
A general nonlinear dissipation could be expressed as 
\begin{equation}
\mathcal{D}_{\rm chem}[{\bm R}, \dot{\bm R}] = \left( {\bm \Gamma}({\bm R}, \dot{\bm R}), \dot{\bm R}  \right) = \sum_{l = 1}^M \Gamma_l ({\bm R}, \dot{\bm R})  \dot{R}_l \geq 0,
\end{equation}
Notice that 
\begin{equation}
  \frac{\dd}{\dd t} \mathcal{F} = \left(\frac{\delta \mathcal{F}}{\delta {\bm R}}, \dot{\bm R}   \right)  = \sum_{l=1}^M \frac{\delta \mathcal{F}}{\delta R_l}  \dot{R}_{l}.
\end{equation}
which indicates that
\begin{equation}\label{R1}
\Gamma_l ({\bm R}, \dot{\bm R}) = - \frac{\delta \mathcal{F}}{\delta R_l},
\end{equation}
which can be viewed as a nonlinear gradient flow on ${\bm R}$. It is interesting to point out that 
\begin{equation} \label{notation-mu-1} 
\frac{\delta \mathcal{F}}{\delta R_l} = \sum_{i=1}^N \frac{\delta \mathcal{F}}{\delta c_i} \frac{\delta c_i}{\delta R_l} =\sum_{i=1}^N \sigma_i^l \mu_i,
\end{equation}
is known as the chemical \emph{affinity}, where $\mu_i =  \frac{\delta \mathcal{F}}{\delta c_i}$ is the chemical potential of $i-$th species. The chemical affinity is the driving force of the chemical reaction \cite{de1927affinite, de1936thermodynamic, kondepudi2014modern}. The choice of dissipation links reaction rate $\dot{\bm R}$ and the chemical affinity. Different choices of dissipation specify different reaction rates.

A typical form of $\mathcal{D}_{\rm chem}[{\bm R}, \dot{\bm R}]$ is given by \cite{wang2020field}
$$ 
  \mathcal{D}_{\rm chem}[{\bm R}, \dot{\bm R}] = \sum_{l=1}^M \dot{R_l} \ln \left( \frac{\dot{R_l}}{ \eta_l ({\bm c})} + 1  \right), 
$$
where $\eta_l({\bm c})$ is the mobility for the $l$-th reaction. Then (\ref{R1}) becomes
\begin{equation}
  \ln \left( \frac{\dot{R_l}}{ \eta_l (R_l)} + 1  \right) = - \sum \sigma_i^{l} \mu_i, \quad l = 1,2, \ldots M.
\end{equation}
The law of mass action can be derived by choosing $\eta_l({\bm c}) = k_{l}^- {\bm c}^{{\bm \beta}^l}$.


As an illustration, we consider a single chemical reaction
\begin{equation}\label{reaction_3}
\ce{ \alpha_1 X_1 + \alpha_2 X_2 <=>[k_1][k_2] \beta_3 X_3},
\end{equation}
with $\bm{\sigma} = (-\alpha_1, - \alpha_2, \beta_3)^{\rm T}$. According to the previous discussion, the variational approach gives
\begin{equation}\label{Eq_R_3}
  \ln \left( \frac{\pp_t R}{k_{1}^{-} c_3^{\beta_3}} + 1 \right) = - \frac{\pp \mathcal{F}}{\pp R},
\end{equation}
where
\begin{equation} \label{pp_ER}
    \frac{\pp \mathcal{F}}{\pp R} = \sum_{i = 1}^3 \sigma_i \mu_i = \sum_{i = 1}^3 \sigma_i (\ln (c_i) + U_i) 
\end{equation}
is the affinity of this chemical reaction. Let ${\bm c} = (c_1^{\infty}, c_2^{\infty}, c_3^{\infty})^{\rm T}$ be a detailed balance equilibrium of the system satisfying $k_1^{+} (c_1^{\infty})^{\alpha_1} (c_2^{\infty})^{\alpha_2} = k_1^- c_3^{\beta_3}$. Recalling (\ref{U_c_eq}), we have
\begin{equation}
\ln K_{eq} = \ln \left( \frac{k_1^+}{k_1^-} \right) = \ln \left( \frac{c_3^{\beta_3}}{(c_1^{\infty})^{\alpha_1} (c_2^{\infty})^{\alpha_2}} \right) = U_1 + U_2 - U_3 . 
\end{equation}
Hence, from (\ref{Eq_R_3}), we obtain
\begin{equation}
  \begin{aligned}
    \pp_t R &  = k_{1}^{-} c_3^{\beta_3} \left(  \frac{1}{K_{eq}} \frac{ c_1^{\alpha_1} c_2^{\alpha_2}}{ c_3^{\beta_3}} - 1 \right) =  k_{1}^{+}c_1^{\alpha_1} c_2^{\alpha_2} -  k_{1}^{-} c_3^{\beta_3}, 
  \end{aligned}
\end{equation}
which is exactly the law of mass action.


\subsection{EnVarA to reaction-diffusion systems}

The above energetic variational formulation for reaction kinetics with detailed balance enables us to model a large class of reaction-diffusion systems in a unified variational way. In general, for a reaction-diffusion system with $N$ species and $M$ reactions, the concentration ${\bm c} \in \mathbb{R}^N$ satisfies the following kinematics
\begin{equation}\label{Kin_1}
 \pp_t c_i + \nabla \cdot (c_i \uvec_i) = \left( {\bm \sigma} \dot{\bm R} \right)_i, \quad i = 1, 2, \ldots N , 
\end{equation}
where $\uvec_i$ is the average velocity of each species due to its own diffusion,  ${\bm R}(\x, t) \in \mathbb{R}^M$ represents the reaction trajectory defined at $\x$, $\dot{\bm R}$ is the time derivative of ${\bm R}$, and ${\sigma} \in \mathbb{R}^{N \times M}$ being the stoichiometric matrix as defined in section 2.2. 
The reaction-diffusion equation can be modeled through the energy-dissipation law \cite{biot1982thermodynamic, wang2020field}: 
\begin{equation}\label{ED_RD}
    \frac{\dd}{\dd t} \mathcal{F}[c_i]  =  - (2 \mathcal{D}_{\rm mech} + \mathcal{D}_{\rm chem}), 
\end{equation}
where the free energy $\mathcal{F}[c_i]$ is the one given in (\ref{free_energy_U}),  
$2 \mathcal{D}_{\rm mech}$ and $\mathcal{D}_{\rm chem}$ are dissipation terms for the mechanical and reaction parts respectively. It is important to notice that the reaction and diffusion parts share the same free energy. As discussed in section 2.1, $\mathcal{D}_{\rm mech}$ can be taken as
$$
  2 \mathcal{D}_{\rm mech} =  \int_{\Omega} \sum_{i=1}^N \xi_i(c_i) |\uvec_i|^2 \dd \x, 
$$ 
where $\xi_i$ is the friction coefficient. Here we assume that $\xi_i$ only depends on $c_i$ for simplicity. In the meantime, the dissipation for the reaction part is assumed to be
\begin{equation*}
  \mathcal{D}_{\rm chem} =  \int_{\Omega} \sum_{l=1}^M \dot{R_l} \ln \left(  \frac{\dot{R_l}}{\eta_l({{\bm c}({\bm R})})} \right) \dd \x.
\end{equation*}
One can view this system as a nonlinear gradient flow on ${\bm R}$, coupled with generalized diffusions on $c_i$.

We can employ the EnVarA to obtain equations for the reaction and diffusion part respectively, i.e., the ``force balance equation'' of the chemical and mechanical subsystems. Indeed, a direct computation yields 
\begin{equation}
\begin{aligned}  
\frac{\dd}{\dd t} \mathcal{F}[{\bm c}] & = \int \sum_{i=1}^N \frac{\delta \mathcal{F}}{\delta c_i} \pp_t c_i \dd \x = \int \sum_{i=1}^N \frac{\delta \mathcal{F}}{\delta c_i} ( -\nabla \cdot (c_i {\bm u}_i) + ({\bm \sigma} \dot{\bm R})_i )\dd \x \\
& = \int \sum_{i=1}^N  \left( c_i \nabla  \mu_i \cdot {\bm u}_i + c_i \sum_{l=1}^M \sigma_i^l \dot{R_l}  \right) \dd \x = \sum_{i=1}^N \left( c_i \nabla \mu_i, \uvec_i  \right) + \sum_{l=1}^M \left( \sum_{i=1}^N \sigma_{i}^l \mu_i,  \dot{R_l}  \right), \\
\end{aligned}
\end{equation}
which indicates that 
\begin{equation}
  \begin{cases}
    & \xi_i(c_i) \uvec_i =  - c_i \nabla \mu_i, \quad i = 1, 2, \ldots N , \\
    &  \ln \left(  \frac{\dot{R_l}}{\eta_l({{\bm c}({\bm R})})} \right) = - \sum_{i=1}^N \sigma_i^l \mu_i, \quad l = 1, \ldots, M . \\
  \end{cases}
\end{equation}
By taking $\xi_i(c_i) = \frac{1}{D_i}c_i$, we have a reaction-diffusion system
\begin{equation}
\pp_t c_i = D_i \Delta c_i + ({\bm \sigma} \dot{\bm R})_i.
\end{equation}
Here $({\bm \sigma} \dot{\bm R})_i$ is the reaction term, the form of which has been discussed in the last subsection. Other choices of $\xi_i(c_i)$ can result in some porous medium type nonlinear diffusion \cite{liu2019lagrangian}, given by
\begin{equation}\label{RD_final}
\pp_t c_i = \nabla \cdot (D(c_i) \nabla c_i) + ({\bm \sigma} \dot{\bm R})_i,
\end{equation}
where $D(c_i) = \frac{c_i}{\xi_i(c_i)}$ is the concentration-dependent diffusion coefficient.

\section{The numerical method}

In this section, we introduce the operator splitting scheme for the reaction-diffusion system (\ref{RD_final}) with the energy-dissipation law (\ref{ED_RD}). Notice that the dissipation in (\ref{ED_RD}) consists of two parts: a gradient flow type dissipation on the reaction trajectory ${\bm R}$, and dissipations for generalized diffusions. It is natural to apply an operator splitting scheme for the time integration to preserve the variational structures, and make the discrete energy non-increasing. More precisely, the reaction-diffusion system (\ref{RD_final}) can be written as
\begin{equation}\label{reaction-diffusion-1}
{\bm c}_t = \mathcal{A}{\bm c} + \mathcal{B}{\bm c},
\end{equation}
where $\mathcal{A}$ is the reaction operator and $\mathcal{B}$ is the diffusion operator.
We propose the following operator splitting scheme. For simplicity, we present the numerical algorithm on the computational domain $\Omega = (0,1)^3$ with a periodic boundary condition, and $\Delta x = \Delta y = \Delta z = h = \frac{1}{N_0}$ with $N_0$ to be the spatial mesh resolution throughout this section; other computational domains with other boundary conditions and other numerical meshes could be handled in a similar fashion.
Given ${\bm c}^n$ with ${\bm c}_{i,j, k}^n \in \mathbb{R}^N_{+}$, we update ${\bm c}^{n+1}$, via the following two stages:
\begin{enumerate}
  \item  {\bf Reaction stage:} \, 
Starting with ${\bm c}^n$, we solve for ${\bm c}_t = \mathcal{A} {\bm c}$ using a positivity-preserving, energy stable numerical scheme for the reaction kinetics at each mesh point, and get ${\bm c}^{n+1, *}$ satisfying
\begin{equation} 
\mathcal{F}_h ({\bm c}^{n+1, *}) \leq \mathcal{F}_h ({\bm c}^{n}),  \quad 
\mbox{$\mathcal{F}_h ({\bm c}): = \langle {\cal F} ({\bm c}), {\bf 1} \rangle$ is the discrete energy} ,   
     \label{splitting-R}     
  \end{equation} 
in which ${\cal F} ({\bm c})$ is given by~\eqref{free_energy_U}, and $\langle f , g \rangle = h^3 \sum_{i,j,k=0}^{N_0-1} f_{i,j,k} g_{i,j,k}$ denotes the discrete $L^2$ inner product.


\item {\bf Diffusion stage:} \, Starting with the intermediate variables 
${\bm c}^{n+1, *}$, 
we update ${\bm c}^{n+1}$ by a suitable positive preserving, and energy stable numerical scheme for diffusion equations such that
\begin{equation} \label{splitting-D} 
\mathcal{F}_h({\bm c}^{n+1}) \leq \mathcal{F}_h({\bm c}^{n+1, *}) . 
\end{equation}
Therefore, a combination of~\eqref{splitting-R} and \eqref{splitting-D} results in 
\begin{eqnarray} 
    \mathcal{F}_h ({\bm c}^{n+1})  \le \mathcal{F}_h ({\bm c}^{n})  . 
     \label{splitting-Field-B-stability}     
\end{eqnarray} 

\end{enumerate}
In the remainder of this section, we will outline the details of the numerical algorithm  
at each stage. 
The convergence and error estimates will be left to the future works. We focus on a single reversible chemical reaction with detailed balance, given by
\begin{equation}\label{reaction_sec3}
\ce{\alpha_1 X_1 + \ldots \alpha_r X_r <=>[k^+_1][k^-_1] \beta_{r+1} X_{r+1} + \ldots \beta_N X_N},
\end{equation}
where $k^+_1$ and $k^-_1$ are constants. The proposed numerical scheme can be applied to reaction-diffusion equations involving $M$ reactions with detailed balance condition directly, but a theoretical justification for the reaction stage will be much more complicated. 


\subsection{Numerical scheme for the reaction stage}
We first introduce a positivity-preserving, energy stable scheme for reaction kinetics, which will be used for the reaction stage in the operator scheme.
Many numerical schemes for reaction kinetics have been developed in the existing literature, such as semi-implicit Runge-Kutta~\cite{alexander1977diagonally}, Patankar-type methods \cite{huang2019positivity, patankar2018numerical}; we refer the interested readers to \cite{formaggia2011positivity} for detailed discussions.

Unlike these existing algorithms, the numerical scheme presented here is based on the energy-dissipation law (\ref{ED_R}). The main idea is to discretize the reaction trajectory ${\bm R}$ instead of the concentration field ${\bm c}$ , so that the conservation property (\ref{conserv}) is automatically satisfied.
Recall that the energy-dissipation law of (\ref{reaction_sec3}) for the reversible chemical reaction (\ref{reaction_sec3}) can be formulated as 
\begin{equation}
\frac{\dd}{\dd t} \mathcal{F} [{\bm c}(R)] = - \dot{R} \ln \left( \frac{\dot{R}}{\eta({\bm c}(R))} + 1 \right),
\end{equation}
in a spatial homogeneous case, where $$\mathcal{F} [{\bm c}] = \sum_{m=1}^N c_m (\ln c_m - 1) + c_m U_m,$$
${\bm c}(R) = {\bm c}_{0} + {\bm \sigma} R$ with ${\bm \sigma} = (- \alpha_1, \ldots -\alpha_r, \beta_{r+1}, \ldots \beta_{N})^{\rm T}$, and $\eta( {\bm c} (R)) = k_1^{-} \prod_{m = r+1}^N c_i^{\beta_i}$.
As shown in section 2.2, the equation of $R$ can be derived as
\begin{equation}\label{eq_R_sec3}
\ln \left(  \frac{\dot{R}}{ \eta( {\bm c} (R) )} + 1 \right) = - \sum_{m=1}^N \sigma_m \mu_m (R),
\end{equation}
where $\mu_m (R) = \frac{\delta \mathcal{F}}{\delta c_m} = \ln c_m (R) + U_m$ is the chemical potential of $m$-th species determined by $R$.
Based on \eqref{eq_R_sec3}, we propose the following semi-implicit scheme
\begin{equation}\label{scheme_R}
\ln \left(  \frac{R^{n+1} - R^n}{ \eta( {\bm c} (R^{n})) \Delta t} + 1 \right) =  - \sum_{m=1}^N \sigma_m \mu_m(R^{n+1}),
\end{equation}
which takes the mobility $\eta ({\bm c}(R))$ explicitly and other parts implicitly. An explicit treatment for $\eta ({\bm c}(R))$ is crucial for the following unique solvability analysis.
\begin{theorem} \label{Positivity: stage 1} 
  Given $R^n$, with ${\bm c}(R^n) \in \mathbb{R}_{+}^N$ and $\eta( {\bm c} (R^n) ) > 0$, there exists a unique solution $R^{n+1}$ for the numerical scheme (\ref{scheme_R}) in the admissible set
  $$\mathcal{V}^n = \{ R ~|~  {\bm c}(R) \in \mathbb{R}_{N}^+, R - R^n + \eta( {\bm c} (R^n) ) \dt > 0 \} , 
$$
so that ${\bm c}(R^{n+1}) \in \mathbb{R}_{+}^N$ and the numerical scheme is well-defined.
\end{theorem}

\begin{proof}
We denote 
  \begin{equation}\label{def_J}
    J_n(R) =   (R - R^n + \eta( {\bm c} (R^n) ) \dt ) \ln \left(  \frac{R - R^n}{ \eta( {\bm c} (R^{n}) ) \dt} + 1 \right) - R +  \mathcal{F}[{\bm c}(R)].
  \end{equation}
  It is easy to see that $J(R)$ is a strictly convex function over $\mathcal{V}^n$. Indeed, a direct computation shows that
  \begin{equation}
    \frac{\dd^2}{\dd R^2} (R - R^n + \eta( {\bm c} (R^n) ) \dt ) \ln \left(  \frac{R - R^n}{ \eta( {\bm c} (R^{n}) ) \dt}  + 1 \right) = \frac{1}{R - R^n + \eta( {\bm c} (R^n) )} > 0.
  \end{equation}
Since $\mathcal{F}[{\bm c}(R)]$ is strictly convex function of ${\bm c} \in \mathbb{R}_N^{+}$, it is clear that $J_n(R)$ is strictly convex over $\mathcal{V}^n$.

  The key point of the proof is to show that the minimizer of $J_n(R)$ over $\mathcal{V}^n$ could not occur on the boundary of $\mathcal{V}^n$, so that a minimizer gives a solution of a numerical scheme (\ref{scheme_R}) in $\mathcal{V}^n$.
 To prove this fact, we consider the following closed domain: 
\begin{eqnarray} 
  \mathcal{V}_{\delta}^n = \left\{ R ~|~ c_i(R) \ge \delta, \quad R - R^n + \eta( {\bm c} (R^n) ) \dt \ge \delta \right\} \subset \mathcal{V}. 
  \label{Field-positive-2} 
\end{eqnarray}   
In the case with single chemical reaction, since
\begin{equation}
  \begin{aligned}
    & c_m^0 + \sigma_m R \ge \delta \Rightarrow   \max_{r+1 \leq m \leq N } \left( \frac{1}{\beta_m} (\delta - c_m^0) \right)   \le   R  \le  \min_{1 \leq m \leq r} \left( - \frac{1}{\alpha_m} (\delta - c_m^0) \right) , 
  \end{aligned}
\end{equation}
it is easy to show that
\begin{equation*} 
\begin{aligned} 
  &
  V_{\delta}^n = \left(  R^n - \eta( {\bm c} (R^n) ) \Delta t + \delta \le R \le  \min_{1 \leq m \leq r} \frac{1}{\alpha_m} (c_m^0 - \delta)   \right) ~\text{or}~  \\
  & V_{\delta}^n = \left(  \max_{r+1 \leq m \leq N} \frac{1}{\beta_m} (\delta - c_m^0) \le R \le \min_{1 \leq m \leq r} \frac{1}{\alpha_m} (c_m^0 - \delta)   \right) . 
\end{aligned} 
\end{equation*}
Thus, $\mathcal{V}_{\delta}^n$ is a bounded, compact set in the real domain and there exists a minimizer of $J_n (R)$ over $\mathcal{V}_{\delta}^n$. We need to show that such a minimizer could not occur on the boundary points of $\mathcal{V}_\delta^n$ if $\delta$ is small enough. 

Assume the minimizer of $J_n (R)$ occurs at a boundary point of $\mathcal{V}_{\delta}^n$. Without loss of generality, we set the minimization point as $R^*= R^n - \eta( {\bm c} (R^n) ) \Delta t + \delta$. A careful calculation reveals the following derivative
\begin{eqnarray} 
 J'_n (R) \mid_{R=R^*}
  &=&  \ln \Big(  \frac{\delta - \eta( {\bm c} (R^n) ) \Delta t}{\eta( {\bm c} (R^n) ) \Delta t}  + 1 \Big) + \sum_{m=1}^N \sigma_m \mu_m (R^*)  \nonumber \\
  &=&  \ln  \delta - \ln (\eta( {\bm c} (R^n) ) \Delta t) + \sum_{m=1}^N\sigma_i \ln c_m (R^*) + \sum_{m=1}^N\sigma_m U_m   \nonumber 
\\
 &\le& 
 \ln  \delta  + \sum_{m=1}^N\sigma_m \ln c_m (R^n - \eta( {\bm c} (R^n) ) \Delta t)  + Q^{(0)},
   \label{Field-positive-4} 
\end{eqnarray} 
where $Q_0 = - \ln (\eta({\bm c} (R^n) ) \Delta t) + \sum_{m=1}^N \sigma_m U_m$.
Since $\sum_{m=1}^N\sigma_m \ln c_m (R^n - \eta( {\bm c} (R^n) ) \Delta t)$ could be taken as a fixed constant, at a fixed time step, we can choose $\delta$ small enough such that
\begin{eqnarray} 
  J'_n (R) \mid_{R=R^*} < 0, 
\end{eqnarray} 
As a result, there exists $R^{**} =  R^n - \eta( {\bm c} (R^n) ) \Delta t + \delta + \delta' \in V_\delta^n$ such that
\begin{eqnarray} 
  J_n (R^{**}) < J_n (R^{*}) .  \label{Field-positive-14} 
\end{eqnarray} 
This makes a contradiction to the assumption that $J (R)$ reaches a minimization point at $R^*$ over $V_{\delta}^n$. 
Using similar arguments, with $\delta$ sufficiently small, if $R^*= \min_{1 \leq m \leq r} \frac{1}{\alpha_m} (c_m^0 - \delta)$, we have 
$$ 
  J'_n (R) \mid_{R=R^*}  > 0 .  
$$ 
Meanwhile, if $ \max_{r+1 \leq m \leq N} \frac{1}{\beta_m} (\delta - c_m^0)$, we will get $ J'_n (R) \mid_{R=R^*}  < 0$. 

As a result, since $J_n (R)$ is a strictly convex function over $\mathcal{V}_n$, the global minimum of $J_n (R)$ over $V_{\delta}^n$ could only possibly occur at an interior point, if $\delta$ is sufficiently small. We conclude that there exists a global minimizer $R^* \in (V_{\delta}^n)^{\mathrm{o}}$, the interior region of $V_{\delta}^n$, of $J_n(R^*)$, so that $J'_n (R) =0$, i.e., $R^*$ is the numerical solution of \eqref{scheme_R}. 
In addition, since $J_n (R)$ is a strictly convex function over $\mathcal{V}_n$, the uniqueness analysis for this numerical solution is straightforward. This completes the proof of Theorem~\ref{Positivity: stage 1}. 
\end{proof} 

\begin{remark} \label{rem: reaction solver} 
  The theorem states that there exists a unique global minimizer of $J_n(R)$ over $\mathcal{V}^n$, which is a solution of the nonlinear scheme (\ref{scheme_R}). 
  In the numerical implementation, it is important to choose a proper optimization method and initial condition, such that the unique minimizer of $J_n(R)$ in $\mathcal{V}^n$ can be found. In practice, we take the initial guess as $R^{n+1, 0} = \Delta t ( k_1^{+} \prod_{m = 1}^r (c_m^{n})^{\alpha_i} - k_1^{-} \prod_{m = r+1}^N (c_m^n)^{\beta_m} ) + R^n$ and use the gradient descent with Barzilai-Borwein (BB) step-size \cite{barzilai1988two} to solve the optimization problem. A backtracking line search is applied to ensure that $R^{n+1, k} \in \mathcal{V}^n$ at $k$-th iteration.
\end{remark}

The numerical scheme (\ref{scheme_R}) can be used in the operator splitting scheme by applying it to each mesh point at the reaction stage. With the positivity-preserving and unique solvability, we can prove its unconditional energy stability \eqref{splitting-R}.
	\begin{theorem} \label{Field-energy stability} 
    Given ${\bm c}^n$, with ${\bm c}_{i,j,k}^n \in \mathbb{R}^N_{+}$ at all mesh points. By setting ${\bm c}^{0} = {\bm c}^n$ and $R^n = 0$, and applying \eqref{scheme_R} at each mesh point, we have
    the following energy dissipation 
        \begin{equation} 
      {\cal F}_h [ {\bm c} (R^{n+1}) ] \le {\cal F}_h [ {\bm c} (R^n) ]. 
         \label{Field-energy-0}     
      \end{equation} 
  where $\mathcal{F}_h ({\bm c}): = \langle {\cal F} ({\bm c}), {\bf 1} \rangle$ is the discrete energy with ${\cal F} ({\bm c})$ is given by~\eqref{free_energy_U}, and ${\bm c}(R) = {\bm c}^0 + {\bm \sigma} R$.
\end{theorem}

\begin{proof} 
Multiplying~\eqref{scheme_R} by $R^{n+1}_{i,j,k} - R^n_{i,j,k}$ at each mesh point, and taking a discrete summation over $\Omega$ yields
\begin{eqnarray} 
\begin{aligned} 
  & 
  \ln \left( \frac{R^{n+1}_{i,j,k} - R^n_{i,j,k}}{ \eta( {\bm c}_{i,j,k} (R^n_{i,j,k}) ) \dt} + 1 \right)  \cdot  
    ( R^{n+1}_{i,j,k}  - R^n_{i,j,k} )  
    +  \sum_{m=1}^N \sigma_m \mu_m (R^{n+1}_{i,j,k}) \cdot ( R^{n+1}_{i,j,k} - R^n_{i,j,k} ) = 0,
\\
  &
   \mbox{so that} \quad 
   \Big\langle \ln \left( \frac{R^{n+1} - R^n}{ \eta( {\bm c} (R^n) ) \dt} + 1 \right)  , 
    R^{n+1}  - R^n \Big\rangle 
    + \Big\langle \sum_{m=1}^N \sigma_m \mu_m (R^{n+1}), R^{n+1} - R^n \Big\rangle = 0, 
\end{aligned} 
  \label{Field-energy-1} 
\end{eqnarray}
in which we have recalled the definition for the discrete inner product, $\langle f , g \rangle = h^3 \sum_{i,j,k=0}^{N_0 -1} f_{i,j,k} g_{i,j,k}$, 
while the second equation comes from a discrete summation of the first equality. And also, $\mu_m =  \frac{\delta \mathcal{F}}{\delta c_m}$ stands for the chemical potential of $m-$th species, as the notation consistent with~\eqref{notation-mu-1}.
The non-negative feature of the first inner product term is straightforward: 
\begin{equation}\label{dissipation-1}
  \Big\langle \ln \Big( \frac{R^{n+1} - R^n}{ \eta( {\bm c} (R^{n}) ) \dt} + 1 \Big) , 
    R^{n+1}  - R^n \Big\rangle \ge 0 , 
\end{equation}    
which comes from the fact that   
   $\ln ( \frac{a - b}{\kappa} + 1 ) \geq 0$ if $a -b \ge 0$ and $\ln ( \frac{a - b}{\kappa} + 1)  \le 0$ if $a - b \le 0$ for $\forall a, b$ and $\forall \kappa > 0$.
On the other hand, by the convexity inequality, $x_1 (\ln x_1 - 1) - x_0 (\ln x_0 - 1) = \int_{x_0}^{x_1} (\ln x ) dx \leq  \ln x_1  (x_1 -x_0)$, we have
\begin{equation*}
  \begin{aligned}
& - \alpha_m ( \ln (c_m^0 - \alpha_m R^{n+1}) + U_m) \cdot ( R^{n+1} - R^n )  \geq (c_m^0 - \alpha_m R^{n+1})\left( \ln (c_m^0 - \alpha_m R^{n+1})  - 1 + U_m \right)  \\
& \qquad \qquad \qquad  \qquad  \qquad  \qquad \qquad  \quad  \quad -  (c_m^0 - \alpha_m R^{n})\left( \ln (c_m^0 - \alpha_m R^{n})  - 1 + U_m \right), \quad m = 1, \ldots r , \\
& ~~ \beta_i ( \ln (c_m^0 + \beta_m R^{n+1}) + U_m ) \cdot ( R^{n+1} - R^n )  \geq (c_m^0 + \beta_m R^{n+1})\left( \ln (c_m^0 + \beta_m R^{n+1})  - 1 + U_m \right)  \\
& \qquad \qquad \qquad  \qquad  \qquad  \qquad \qquad  \quad  \quad -  (c_m^0 + \beta_m R^{n})\left( \ln (c_m^0 + \beta_m R^{n})  - 1 + U_m \right), \quad m = r+1, \ldots N . 
  \end{aligned}
\end{equation*}
at each mesh point (here we omit the index $(i,j,k)$ for simplicity).
In turn, a discrete summation of the above estimates leads to
\begin{equation}\label{energy_1}
  \Big\langle \sum_{m=1}^N \sigma_m \mu_m (R^{n+1}), R^{n+1} - R^n \Big\rangle 
  \geq \mathcal{F}_h [{\bm c}(R^{n+1})] - \mathcal{F}_h [{\bm c}(R^{n})] . 
\end{equation}
Again, the notation of discrete inner product, $\langle f , g \rangle = h^3 \sum_{i,j,k=0}^{N_0 -1} f_{i,j,k} g_{i,j,k}$, as well as the definition for the discrete energy, $\mathcal{F}_h ({\bm c})  = \langle {\cal F} ({\bm c}), {\bf 1} \rangle$ (given by~\eqref{splitting-R}), have been recalled.In turn, a substitution of~\eqref{dissipation-1} and \eqref{energy_1} into~\eqref{Field-energy-1} leads to~\eqref{Field-energy-0}, which proves unconditional energy stability. 

\end{proof}

For a general reversible reaction network with $M$ reaction and $N$ reaction, we can construct a similar numerical scheme, given by
\begin{equation}\label{scheme_Mul_R}
 \ln \left(  \frac{R_l^{n+1} - R_l^n}{ \eta_l({\bm c} ({\bm R}^n)) \dt} + 1 \right) = - \sum_{m=1}^N \sigma_m^l \mu_m^{n+1}, \quad l = 1, 2, \ldots M,
\end{equation}
where $\mu_m^{n+1}$ is the chemical potential of $m-$th species for given ${\bm R}^{n+1}$.
The corresponding optimization problem can be defined by
\begin{equation}\label{scheme_R_op}
  \begin{aligned}
&  {\bm R}^{n+1} = \text{argmin}_{{\bm R} \in \mathcal{V}^n} J (\bm{R})   \\
 & J({\bm R}) =  \sum_{l=1}^M \left( (R_l - R^n_l + \eta_l ({\bm c} (R^n_l) ) \dt ) \ln \left(  \frac{R_l - R^n_l}{ \eta_l ( {\bm c} ({\bm R}^n) \dt} + 1 \right) - R_l \right)  + \mathcal{F} [ {\bm c} ({\bm R})] , 
  \end{aligned}
\end{equation}
where ${\bm c}({\bm R}) = {\bm c}_0 + {\bm \sigma} {\bm R}$, and
\begin{equation}
\mathcal{V}^n = \{  {\bm R} \in \mathbb{R}^M ~|~ {\bm c} ({\bm R}) \in \mathbb{R}^N_{+}, \quad R_l - R^n_l + \eta_l ({\bm c} ({\bm R}^n) ) \dt > 0 \}.
\end{equation}
It is easy to show that $J(\bm{R})$ is strictly convex over $\mathcal{V}^n$.
Following the same arguments, it would be possible to establish the positivity-preserving property, unique solvability, and energy stability when ${\rm rank} \, ( {\bm \sigma} )= M$ and $M < N$.  However, the technical details will be much more complicated than the single reaction case. In this paper, we will show that the numerical scheme (\ref{scheme_Mul_R}) works well practically for multiple reactions case through a numerical example. The theoretical justification for numerical scheme (\ref{scheme_Mul_R}) will be explored in the future works.

\begin{remark}
The numerical scheme (\ref{scheme_R_op}) is similar to the mirror descent in optimization. For an optimization problem $\min_{{\bm R}}\mathcal{F}({\bm R})$, the implicit mirror descent can be formulated as
\begin{equation}
R^{n+1} = \mathop{\arg\min}_{{\bm R} \in \mathcal{V}} \frac{1}{\eta} D_{\varphi} ({\bm R}^n, {\bm R}) + \mathcal{F}({\bm R}),
\end{equation}
where $\mathcal{V}$ is the admissible set, $\eta$ is the step-size, and $D_{\varphi} ({\bm R}, {\bm R}^n)$ is the Bregman divergence induced by $\varphi$
\begin{equation}
  D_{\varphi} (\z, \x) = \varphi(\x) - \varphi(\z) - \langle \nabla \varphi(\z), (\x - \z) \rangle.
\end{equation}
The Bregman $D_{\varphi} ({\bm R}, {\bm R}^n)$ measures the distance between ${\bm R}$ and ${\bm R}^n$. The classical implicit Euler method corresponds to the case that $\varphi(\x) = \frac{1}{2} |\x|^2$.
Recall the definition of $J({\bm R})$ in (\ref{scheme_R_op}), which can be reformulated as
\begin{equation}
J({\bm R}) = \frac{1}{\Delta t}\Psi \left( {\bm R}^*, \frac{{\bm R} - {\bm R}^n}{\Delta t} \right) + \mathcal{F}[{\bm c}({\bm R})],
\end{equation}
where $\Psi \left( {\bm R}^*, \frac{{\bm R} - {\bm R}^n}{\Delta t} \right)$ corresponds to a Bregman divergence that measures the distance between ${\bm R}$ and ${\bm R}^n$. Notice that
\begin{equation}
\frac{1}{\Delta t} \frac{\delta \Psi \left( {\bm R}^*, \frac{{\bm R} - {\bm R}^n}{\Delta t} \right) }{\delta {\bm R}} = {\bm \Gamma} ({\bm R}^*, \frac{{\bm R} - {\bm R}^{n}}{\Delta t}) . 
\end{equation}
Under a certain limit, we have
\begin{equation}
\frac{\delta \Psi}{\delta \dot{{\bm R}}} = {\bm \Gamma} ({\bm R}, \dot{{\bm R}}),
\end{equation}
so the dynamics for the reaction kinetics can be reformulated as
\begin{equation}
\frac{\delta \Psi({\bm R}, \dot{\bm R})}{\delta \dot{\bm R}} = - \frac{\delta \mathcal{F}}{\delta {\bm R}}.
\end{equation}
One can view the left-hand side as generalized Onsager's maximum dissipation principle and $\Psi({\bm R}, \dot{\bm R})$ is the pseudo-potential of dissipative forces \cite{mielke2015rate}.
For given energy dissipation $\triangle(R, \dot{R})$, one way to define $\Psi(R, \dot{R})$ is given by~\cite{mielke2015rate} 
\begin{equation}
  \Psi({\bm R}, \dot{\bm R}) = \int_{0}^1   \frac{1}{\theta} \triangle({\bm R}, \theta \dot{\bm R}) \dd \theta,
\end{equation}
which implies that 
\begin{equation}
\left( \frac{\delta \Psi}{\delta \dot{\bm R}}, \dot{\bm R} \right) =   \int_{0}^1 \left(  \frac{\delta \triangle ({\bm R}, \theta \dot{\bm R}) }{\delta \dot{\bm R}}, \dot{\bm R} \right)\dd \theta  = \triangle ({\bm R}, \dot{\bm R}) - \triangle({\bm R}, 0) =  \triangle ({\bm R}, \dot{\bm R}) . 
\end{equation}
In general, an explicit form of $\Psi({\bm R}, \dot{\bm R})$ may not be available.
\end{remark}


\begin{remark} 
The unconditional energy stability of the proposed scheme~\eqref{scheme_R} follows the idea from the convex-concave decomposition of the energy, an idea popularized in Eyre's work~\cite{eyre98}. The method has been applied to the phase field crystal (PFC) equation~\cite{guan16a, wise2009energy};  epitaxial thin film growth models~\cite{chen12, wang10a}; non-local gradient model~\cite{guan14a}; the Cahn-Hilliard model coupled with fluid flow~\cite{chen16, liuY17}, etc. Second-order accurate energy stable schemes have also been reported in recent years, based on either the Crank-Nicolson or BDF2 approach~\cite{baskaran13a, baskaran13b, diegel17, diegel16, guan14b, guo16, hu09, shen12, yan18}, etc. In fact, the discrete energy functional, introduced in~\eqref{splitting-R}, is convex in terms of ${\bm c}$, and no concave term is involved. As a result, all the logarithmic terms are implicitly treated, for the sake of both the energy stability and positivity-preserving property. The implicit and singular nature of these terms prevents the numerical solution reaching the limiting values, as demonstrated in the analysis. Meanwhile, the denominator term $\eta ({\bm c} (R))$, which appears in $\ln \Big( \frac{R_r}{\eta ( {\bm c} (R) ) } +1 \Big)$, is treated explicitly, which plays a similar role as a mobility. Such an explicit treatment ensures the convexity of the corresponding numerical approximation associated with the temporal discretization, which makes the theoretical analysis pass through. 
\end{remark}

\subsection{Numerical scheme for the diffusion stage} 

It is noticed that for the cases that we consider, the equations for each species are fully decoupled at the diffusion stage ${\bm c}_t = \mathcal{B} {\bm c}$. Henceforth, we only need to construct a positivity-preserving, energy stable scheme for a single diffusion equation
\begin{equation}
\rho_t = \nabla \cdot  (D(\rho) \nabla \rho). 
\end{equation}

There have been extensive existing works of positivity-preserving and energy stable scheme for generalized diffusion equations \cite{liu2020positivity, shen2020PNP, gu2020bound, carrillo2018lagrangian, liu2019lagrangian}.  
For the diffusion equation that considered in this paper, we can adopt the classical semi-implicit scheme
\begin{equation}\label{semi_imp}
  \frac{\rho^{n+1} - \rho^n}{\dt} = \nabla_h \cdot \left( {\cal A}_h [D (\rho^n)]  \nabla_h \rho^{n+1}  \right),
\end{equation}
where $\nabla_h$ and $\nabla_h \cdot$ stand for the discrete gradient and the discrete divergence, respectively. The spatially averaged field ${\cal A}_h [D (\rho^n)]$, which is introduced to obtain the value of $D(\rho^n)$ at the staggered mesh points, is defined as 
\begin{equation} 
\begin{aligned} 
  & 
  ( {\cal A}_h [D (\rho^n)] )_{i+1/2, j,k} = \frac12 ( ( D (\rho^n) )_{i, j,k} + ( D (\rho^n) )_{i+1, j,k}  ),  
\\
  &
  ( {\cal A}_h [D (\rho^n)] )_{i, j+1/2,k} = \frac12 ( ( D (\rho^n) )_{i, j,k} + ( D (\rho^n) )_{i, j+1,k}  ),  
\\
  &
  ( {\cal A}_h [D (\rho^n)] )_{i, j,k+1/2} = \frac12 ( ( D (\rho^n) )_{i, j,k} + ( D (\rho^n) )_{i, j,k+1}  ) .    
\end{aligned} 
\label{average-1} 
\end{equation}    
The key point of such an average operator is associated with the feature that ${\cal A}_h [D (\rho^n)] >0$ at all staggered mesh points, provided that $D (\rho^n)_{i,j,k} > 0$. For the standard semi-implicit scheme (\ref{semi_imp}), its positivity-preserving property could be theoretically developed, as well as the energy stability, where the logarithmic function in the free energy plays an essential role. 



Indeed, the positivity-preserving property for this numerical scheme is obvious, due to the discrete maximum principle. 
\begin{theorem}  \label{Heat-positivity} 
  Given $\rho^n$, with $\rho_{i,j,k}^n > 0$, $0 \le i,j,k \le N_0$, there exists a unique solution $\rho^{n+1}$ for the numerical scheme~\eqref{semi_imp}, with the discrete periodic boundary condition, with $\rho_{i,j,k}^{n+1} > 0$, $0 \le i,j,k \le N_0$.   
\end{theorem} 

To prove the energy stability, we first introduce the discrete free energy for any discrete grid function $\rho$:
\begin{equation} 
  \mathcal{F}_h (\rho) := \langle \rho \ln \rho + C \rho ,  {\bf 1} \rangle, 	\label{energy-phase-discrete-1} 
\end{equation} 
where $C$ is an arbitrary constant, $\langle f , g \rangle = h^3 \sum_{i,j,k=0}^{N_0-1} f_{i,j,k} g_{i,j,k}$ denotes the discrete $L^2$ inner product. Thanks to the positivity-preserving property and unique solvability for the numerical scheme~\eqref{semi_imp}, we are able to establish an unconditional energy stability. 
  
	\begin{theorem}
	\label{Heat-energy stability} 
For the numerical solution~\eqref{semi_imp},  with the discrete periodic boundary condition, we have $\langle \rho^{n+1}, {\bf 1} \rangle = \langle \rho^{n}, {\bf 1} \rangle$ and
	\begin{equation} 
\mathcal{F}_h (\rho^{n+1}) \le \mathcal{F}_h (\rho^n),   \label{heat-energy-0} 
	\end{equation} 
so that $\mathcal{F}_h (\rho^n) \le \mathcal{F}_h (\rho^0)$, an initial constant. 
	\end{theorem}

\begin{proof} 

The conservation property can be proved by taking a discrete inner product with~\eqref{semi_imp} by ${\bf 1}$, which results in
\begin{equation}
\langle \rho^{n+1} - \rho^n, {\bf 1} \rangle = 0,
\end{equation}
due to the discrete periodic boundary condition. Moreover, taking a discrete inner product with~\eqref{semi_imp} by $\mu^{n+1} = \ln \rho^{n+1} + C$ yields  
\begin{eqnarray} 
\begin{aligned} 
  \frac{1}{\tau} \langle \mu^{n+1} + C, \rho^{n+1} - \rho^n \rangle  
  =  & \langle \mu^{n+1} , 
  \nabla_h \cdot \left( {\cal A}_h[D(\rho^{n})] \nabla_h \rho^{n+1}  \right)  \rangle  
\\
  =& 
  - \langle \nabla_h (\ln \rho^{n+1}), {\cal A}_h[D(\rho^n)] \nabla_h \rho^{n+1} \rangle_{*} 
  \le 0 \label{heat-energy-1}, 
\end{aligned} 
\end{eqnarray}
where the last step follows the inequality $(a - b) \ln (a - b) \geq 0$, combined with the fact that ${\cal A}_h[D(\rho^n)]  > 0$. Here $\langle \cdot, \cdot \rangle_{*}$ denotes the discrete inner product defined on the staggered mesh points.

Due to convexity of the discrete energy $\mathcal{F}_h (\rho) = \langle \rho \ln \rho + C \rho, {\bf 1} \rangle$, we have
\begin{equation} 
\mathcal{F}_h (\rho^{n+1}) -  \mathcal{F}_h (\rho^{n}) \leq \langle \mu^{n+1}, \rho^{n+1} - \rho^n \rangle \leq 0 . \label{heat-energy-2}
\end{equation}
In turn, a substitution of~\eqref{heat-energy-2} into \eqref{heat-energy-1} leads to~\eqref{heat-energy-0}, so that the unconditional energy stability is proved. 
\end{proof} 

\begin{remark}
  Such an energy stability analysis is available when the free energy is convex with respect to $\rho$, so that $\mu(\rho)$ is increasing in $\rho$. The discrete energy-dissipation law corresponding (\ref{semi_imp}) becomes 
  \begin{equation}
    \frac{\mathcal{F}^{n+1} - \mathcal{F}^n}{\dt} = -  \left \langle  {\cal A}_h \left[ \frac{\widebreve{\cal M}^n}{\rho^n}  \right] \nabla_h \rho^{n+1}, \nabla_h \mu^{n+1}   \right\rangle \leq 0 . 
    \end{equation}
\end{remark}



\begin{remark} 
Although the stability for the free energy (\ref{energy-phase-discrete-1}) can be proved for the semi-implicit scheme, the algorithm is not based on the variational structure associated with this energy. 
In fact, the semi-implicit scheme corresponds to the minimization problem
\begin{equation}
  \rho = \mathop{\arg\min}_{\rho} \frac{1}{2 \dt} \| \rho - \rho^n \|_2^2 +     \int D(\rho^n) \nabla \rho \cdot \nabla \rho ~ \dd \x . 
\end{equation}
A variational structure preserving scheme for a diffusion equation with the free energy (\ref{energy-phase-discrete-1}) can be developed by using the Lagrangian method \cite{carrillo2018lagrangian, liu2019lagrangian} or the numerical methods for Wasserstein gradient flows \cite{benamou2016augmented}.
\end{remark}

\begin{remark}
Recently, there have been various works devoted to designing a positivity-preserving and energy stable scheme for generalized diffusion systems, such as Poisson-Nernst-Planck equations \cite{liu2020positivity, shen2020PNP} and the porous medium equations \cite{gu2020bound}. 
The basic idea is to use the numerical scheme
\begin{equation}\label{D_scheme2}
\frac{\rho^{n+1} - \rho^n}{\dt}  = \nabla_h \cdot 
\left( {\cal A}_h [\widebreve{\cal M}^n] \nabla_h \mu^{n+1}  \right),  
\end{equation}
where $\nabla_h$ and $\nabla_h \cdot$ are the discrete gradient and the discrete divergence, 
${\cal A}_h [\widebreve{\cal M}^n]$ is the spatial average of the mobility at the staggered mesh points, and $\mu^{n+1}$ is the chemical potential taken in the $n+1$-th time step. 
By taking the mobility explicitly, the proof of unique solvability, positivity-preserving, as well as energy stability, follows through the convexity analysis for the logarithmic terms in the free energy. Moreover, if $\mathcal{F}_h$ is a convex energy, the numerical solution (\ref{D_scheme2}) satisfies the following discrete energy-dissipation law 
\begin{equation}
\frac{\mathcal{F}^{n+1}_h - \mathcal{F}^n_h}{\dt} = -  \langle {\cal A}_h [\widebreve{\cal M}^n] \nabla_h \mu^{n+1}, \nabla_h \mu^{n+1}   \rangle \leq 0 . 
\end{equation}
An interesting fact we have discovered here is that, if the physical energy is convex, such as the porous medium type diffusion with $\rho \ln \rho$ as the physical energy, it is not necessary to use the nonlinear scheme (\ref{D_scheme2}). Instead, a standard semi-implicit linear scheme can achieve a similar theoretic goal. The idea of energy-dissipation law based operator splitting approach can be applied far beyond reaction-diffusion systems.
\end{remark}



\subsection{The operator splitting scheme for the reaction-diffusion system} 
Based on the previous discussion for the reaction-diffusion system~\eqref{equation: reaction-diffusion}, the following operator splitting scheme is proposed. For simplicity of presentation, only the case of a single reversible chemical reaction is considered, and an extension to the case of multiple reactions is expected to be straightforward. 

Given ${\bm c}^n$, with ${\bm c}_{i,j,k}^n \in \mathbb{R}^N_{+}$ at all mesh points. We update ${\bm c}^{n+1}$, via the following two stages. 

\noindent 
{\bf Stage 1.} \, First, we set ${\bm c}^{n, 0} = {\bm c}^n$, and apply the following scheme (corresponding to~\eqref{scheme_R}) at each mesh point:
\begin{eqnarray} 
  && 
\ln \left(  \frac{R^{n+1} - R^n}{\eta( {\bm c}^n)) \Delta t} + 1 \right) = 
- \sum_{m=1}^N \sigma_i \mu_m (R^{n+1}),  \quad R^n \equiv 0,   \label{splitting-Field-A-2}   
\end{eqnarray}
where $\mu_{m}(R^{n+1}) =  {\bm c}^{n, 0} + {\bm \sigma} R^{n+1}$ is the chemical potential of $m$-th species for given $R^{n+1}$. Here we omit the index $(i,j,k)$ for simplicity. 
By Theorem~\ref{Positivity: stage 1}, there exists a unique solution $R^{n+1}_{i,j,k}$ at each mesh point such that $c_{i,j,k} + {\bm \sigma} R^{n+1}_{i,j,k} \in \mathbb{R}^N_{+}$. We denote 
\begin{equation} 
   {\bm c}^{n+1,*}_{i,j,k} := {\bm c}^n_{i,j,k} + {\bm \sigma} R^{n+1}_{i,j,k},
   \label{splitting-Field-A-3}   
\end{equation}    
and we have the  the following energy dissipation 
    \begin{equation} 
  {\cal F}_h ( {\bm c}^{n+1,*}) \le {\cal F}_h ( {\bm c}^n ) . 
     \label{splitting-Field-A-4}     
  \end{equation} 
according to Theorem~\ref{Field-energy stability}, where $\mathcal{F}_h ({\bm c}): = \langle {\cal F} ({\bm c}), {\bf 1} \rangle$ is the discrete energy with ${\cal F} ({\bm c})$ is given by~\eqref{free_energy_U}.

\noindent 
{\bf Stage 2.} \,The intermediate variables ${\bm c}^{n+1,*}$ have been proved to be positive at each mesh point. Next, we update ${\bm c}^{n+1}$ by the semi-implicit scheme (\ref{semi_imp})
\begin{equation} 
  \frac{{\bm c}_m^{n+1} - {\bm c}_m^{n+1,*}}{\dt} = \nabla_h \cdot \left( {\cal A}_h [ D_m ({\bm c}_m^n) ] \nabla_h {\bm c}_m^{n+1} \right) , \, \, \,  
 1 \le m \le N  .  
    \label{splitting-Field-B} 
\end{equation} 
With an application of Theorem~~\ref{Heat-positivity} and \ref{Heat-energy stability}, both the positivity and energy stability are preserved:
\begin{eqnarray} 
  &&  
  {\bm c}^{n+1}_{i,j,k} \in \mathbb{R}_{+}^N, \quad {\cal F}_h ( {\bm c}^{n+1})  
    \le {\cal F}_h ( {\bm c}^{n+1,*}). 
     \label{splitting-Field-B-6}     
\end{eqnarray} 
A combination of~\eqref{splitting-Field-A-4} and \eqref{splitting-Field-B-6} results in 
\begin{eqnarray} 
    {\cal F}_h ( {\bm c}^{n+1})  \le {\cal F}_h ( {\bm c}^n)  . 
     \label{splitting-Field-B-7}     
\end{eqnarray} 

Therefore, we arrive at the following theoretical result for the operator splitting scheme.  
\begin{theorem}  \label{Field-operator splitting} 
  Given ${\bm c}^n$, with ${\bm c}^n_{i,j,k} \in \mathbb{R}_{+}^N$, there exists a unique solution ${\bm c}^{n+1}$, with discrete periodic boundary condition, for the operator splitting numerical scheme~\eqref{splitting-Field-A-2}-\eqref{splitting-Field-A-3}, combined with  \eqref{splitting-Field-B}. The point-wise positivity is ensured, i.e., ${\bm c}^{n+1}_{i,j,k} \in \mathbb{R}_{+}^N$ at all mesh points. In addition, we have the energy dissipation estimate: ${\cal F}_h ( {\bm c}^{n+1})	 \le {\cal F}_h ( {\bm c}^n)$, so that ${\cal F}_h ({\bm c}^n) \le {\cal F}_h ({\bm c}^0)$. 
\end{theorem}

\begin{remark} 
There have been many existing works of operator splitting numerical approximation to nonlinear PDEs, such as \cite{descombes2001convergence, descombes2004operator, zhao2011operator} for reaction-diffusion systems, \cite{Bao2002, besse2002order, Lubich2008, Shen2013, Thalhammer2012} for the nonlinear Schr\"odinger equation, \cite{Badia2013} for the incompressible magnetohydrodynamics system, \cite{BCF2013} for the delay equation, \cite{EO2014a} for the nonlinear evolution equation, \cite{EO2014b} for the Vlasov-type equation, \cite{KV2015} for a generalized Leland's mode, \cite{zhangC18a, zhangC17a} for the ``Good" Boussinesq equation, \cite{Lee2015} for the Allen-Cahn equation,  \cite{Li2017} for the molecular beamer epitaxy (MBE) equation, \cite{Zhao2014} for nonlinear solvation problem, etc. On the other hand, different computational stages correspond to different physical energy for most existing works of operator splitting, so that a combined energy stability estimate becomes very challenging at a theoretical level. In comparison, the same free energy, in the form of ${\cal F}_h ( {\bm c} ) = \langle \mathcal{F}({\bm c}) , {\bf 1} \rangle$, dissipates at both the reaction and the diffusion stages in our proposed scheme, so that the combined operator splitting scheme yields a dissipated energy at each time step, with a careful transformation~\eqref{splitting-Field-A-3} between two splitting stages. In the authors' best knowledge, it is the first time to report a free energy dissipation for a combined operator splitting scheme to a nonlinear PDE with variational structures.  
\end{remark} 

\begin{remark} 
The proposed numerical method is designed for reaction-diffusion equations with variational structure, i.e., the reaction part satisfies the law of mass action with the detailed balance condition. The variational structure is the key point that enables one to develop a positive-preserving and energy stable numerical schemes. The proposed numerical framework is not limited to reaction-diffusion systems, but can be applied to more general dissipative systems with different components of dissipations \cite{knopf2020phase, liu2019energetic, wang2021two}. In the existing literature, extensive higher-order positivity-preserving numerical schemes have been developed for general reaction-diffusion equation \cite{formaggia2011positivity, huang2019positivity,sandu2001positive}, while most of them didn't use the variational structure, so that the energy stability was not theoretically justified. In comparison, it is straightforward to apply our proposed approach to develop high-order schemes for dissipative systems.
\end{remark}

\begin{remark} 
The proposed scheme is very efficient in terms of numerical implementation, for both 2-D and 3-D computations. At the reaction stage, the optimization problem~\eqref{scheme_R_op} corresponds to a convex energy to be optimized, and this fact has greatly facilitated the computational efforts. In fact, a lot of optimization methods have been available for this convex optimization problem. In the current study, we use simple gradient descent with Barzilai-Borwein (BB) step-size and apply a backtracking line search to ensure $R$ is in the admissible set; see the detailed descriptions in Remark~\ref{rem: reaction solver}. In particular, such a numerical solver is point-wise applied over the numerical grid, and there is no coupling between the numerical mesh points, because of the ODE nature of the system~\eqref{scheme_Mul_R} (and \eqref{splitting-Field-A-2}). Therefore, only an order of $O (N_0^d)$ (with $d$ the dimension) calculations are needed in the reaction stage, which corresponds to a computational cost less than a standard Poisson solver.

In terms of the numerical implementation for~\eqref{splitting-Field-B} in the diffusion stage, we observe that it is a linear elliptic equation with given variable coefficients. Many existing numerical solvers, such as conjugate gradient method and preconditioned steepest descent gradient iteration algorithm, could be very efficiently applied. Extensive numerical experiments have demonstrated that, the computational cost of such a numerical solver is at the same level as the Poisson solvers, maybe three to five times of a standard Poisson solver, with a reasonable initial iteration guess. In addition, the availability of FFT-based computational tools has further improved the numerical efficiency. As a result, the proposed numerical algorithm~\eqref{splitting-Field-B} in the diffusion stage has a great advantage in terms of numerical efficiency; it is indeed computationally much more efficient than many other positivity-preserving and energy stable schemes, such as the the scheme in~\cite{liu2020positivity} for the Poisson-Nernst-Planck system and the one in~\cite{gu2020bound} for the porous medium equations.

Since the computational cost at the reaction stage is of order $O (N_0^d)$, even less than a standard Poisson solver, and the computational cost at the diffusion stage is at the same level as the Poisson solvers, the spatial dimension does not pose a great computational challenge to the proposed operator splitting scheme~\eqref{splitting-Field-A-2}-\eqref{splitting-Field-B}. Both 2-D and 3-D computations could be efficiently performed. In this article we focus on the 2-D computational results, and more 3-D modeling and numerical results will be reported in the future works.
\end{remark}

\section{Numerical results} 

In this section, we perform some numerical tests for the proposed numerical scheme to demonstrate the energy stability, as well as the positivity-preserving property. 

\subsection{Reaction kinetic: accuracy test}
We first consider a simple test case, as used in \cite{huang2019positivity}, given by
\begin{equation}\label{test_1}
\begin{cases}
  & \dfrac{\dd c_1}{\dd t} = c_2 - ac_1 , \\
  & \dfrac{\dd c_2}{\dd t} = a c_1 - c_2  ,  
\end{cases}
\end{equation}
where $a > 0$ is a constant. This equation corresponds to a simple reversible chemical reaction $\ce{X_1 <=>[a][1] X_2}$. For a given initial condition $c_i(0) =c_i^0$, the exact solution can be computed as
\begin{equation}
c_1(t) = (1 + \left( \frac{c_1^0}{c_1^{\infty}} -1 \right) \exp(-(a+1)) t) c_1^{\infty}, \quad c_2(t) = c_1^0 + c_2^0 - c_1(t),
\end{equation}
where $c_1^{\infty} = (c_1^0 + c_2^0) / (a+1)$ is the equilibrium concentration of $X_1$. Let $R$ be the reaction trajectory, 
the corresponding energy-dissipation law can be expressed as
\begin{equation}
\frac{\dd}{\dd t} \left( \sum_{i=1}^2 c_i (\ln c_i - 1) + c_1 \ln a + c_2 \ln (1)  \right) = - \dot{R} \ln \left(  \frac{\dot{R}}{c_2} + 1 \right).
\end{equation}

The errors between numerical solution and exact solution at $T = 1$ are displayed in Fig. \ref{Fig1} (a). A clear first-order accuracy is observed. Since the proposed operator splitting scheme only has first-order temporal accuracy, a first-order numerical approximation for the reaction part is adequate for our purpose. Fig. \ref{Fig1} (b) shows the evolution of numerical free energy with respect to time ($\dt = 1/160$), which verifies the energy stability of the scheme.

\begin{minipage}{\textwidth}
  \begin{minipage}[b]{0.4\textwidth}
    \centering
    \begin{tabular}{c|c|c} 
     $\dt$ & Error & Order \\
     \hline
     1/20  & 0.02840 & \\
     1/40  & 0.01377 & 1.04 \\
     1/80  & 6.767e-3 & 1.02\\
     1/160 & 3.353e-3 & 1.01 \\
     1/320 & 1.669e-3 & 1.0066 \\
   \end{tabular}

   \vspace{2em}
  \end{minipage}
  \hspace{1em}
  \begin{minipage}[b]{0.35\textwidth}
    \centering

    \vspace{1em}
    \begin{overpic}[width = \textwidth]{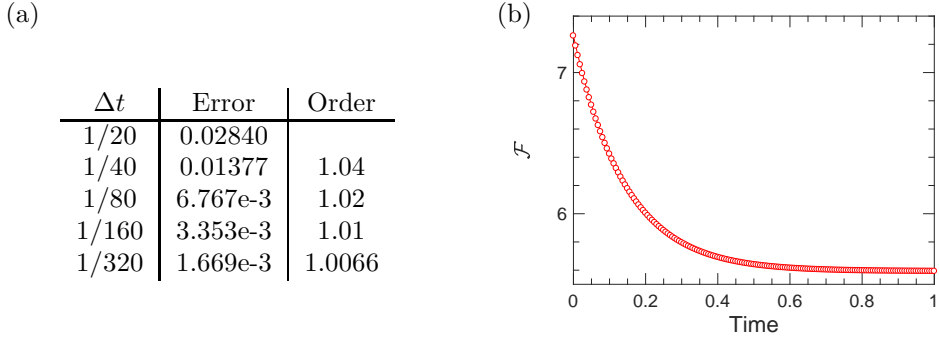}
      \put(-4, 74){(b)}
      \put(-120, 74){(a)}
    \end{overpic}
    \end{minipage}
    \captionof{figure}{(a) Error table for the linear ODE system (\ref{test_1}). (b) Numerical free energy with respect to time ($\dt = 1/160$).}\label{Fig1}
    \vspace{1em}

  \end{minipage}

\subsection{Generalized Michaelis-Menten kinetics}
In this subsection, we test our numerical scheme for a multiple reaction case. Consider a generalized Michaelis-Menten enzyme kinetics with two intermediate states \cite{keener1998mathematical, peller1959multiple}, in which the reactions are given by
\begin{equation}\label{MM_R}
\ce{E + S <=>[k_1^+][k_1^-] ES}, \quad \ce{ES <=>[k_2^+][k_2^-] EP}, \quad \ce{EP <=>[k_3^+][k_3^-] E + P}.
\end{equation}
This is a reaction network with 5 species and 3 reactions, where  ${\rm E}$ is the enzyme that catalyzes the reaction $\ce{S <=> P}$,  ${\rm SE}$ and ${\rm SP}$ are two intermediates. It is often assumed that $k_2^- \ll k_2^+$ and $k_3^- \ll k_3^+$, so that most of ${\rm S}$ will be converted to ${\rm E}$. One can view the reaction network (\ref{MM_R}) as a thermodynamically consistent modification of the classical Michaelis-Menten kinetics $\ce{S + S <=>[k_1^+][k_1^1] ES ->[k_{\rm cat}] P}$, which doesn't have any variational structure \cite{oster1974chemical}. In a modeling perspective, it would be interesting to compare both models and study the singular limit of (\ref{MM_R}) with $k_2^- \rightarrow 0$ and $k_3^- \rightarrow 0$. We are working on this direction.

Let $c_1 = [\rm E]$, $c_2 = [\rm S]$, $c_3 = [\rm ES]$, $c_4 = [\rm EP]$, $c_5 = [\rm P]$ and ${\bm R} = (R_1, R_2, R_3)^{\rm T}$ denote three reaction trajectories, the energy-dissipation law of the system can be formulated as
\begin{equation}
\frac{\dd}{\dd t} \left( \sum_{i=1}^5 c_i (\ln c_i -  1 + U_i)  \right) = - \int \sum_{l=1}^3 \dot{R_l} \ln \left( \frac{\dot{R_l}}{\eta_l ({\bm c}({\bm R})) }  - 1 \right),
\end{equation}
where $U_i$ can be taken as
\begin{equation*}
U_1 = - \ln (k_1^- k_2^- k_3^-), \quad U_2 = \ln (k_1^+ k_3^-) \quad U_3 = -\ln [k_2^-], \quad U_4 = -\ln [k_2^+], \quad U_5 = -\ln (k_1^+ k_2^+ k_3^+), 
\end{equation*}
and
\begin{equation*}
\eta_1({\bm c}) = k_1^- c_3, \quad \eta_2({\bm c}) = k_2^- c_4, \quad \eta_2({\bm c}) = k_3^- c_1 c_5.
\end{equation*}
We can derive the law of mass action for (\ref{MM_R}) under the above choice of parameters.

In the numerical simulation, we take $k_1^+ = 1$, $k_1^- = 0.5$, $k_2^+ = 100$, $k_2^- = 1$, $k_3^+ = 100$ and $k_3^- = 1$. The initial condition is set as $[\rm S](0) = 1$, $[\rm E](0) = 0.8$, $[\rm ES](0) = [\rm EP](0) = [P][0] = 0.01$. Figure~\ref{MM_Res} displays the numerical result with $\dt = 1/50$. This numerical simulation turns out to be very challenging due to the low concentration levels of ${\rm ES}$ and ${\rm EP}$. Indeed, for the parameters we used, $[\rm ES]$ is about $3.57 \times 10^{-4}$ in the late stage of the simulation.
The numerical results have clearly demonstrated the positivity-preserving property and energy stability of our proposed numerical scheme in the case with multiple reactions.

\begin{figure}[!h]
  \includegraphics[width = \linewidth]{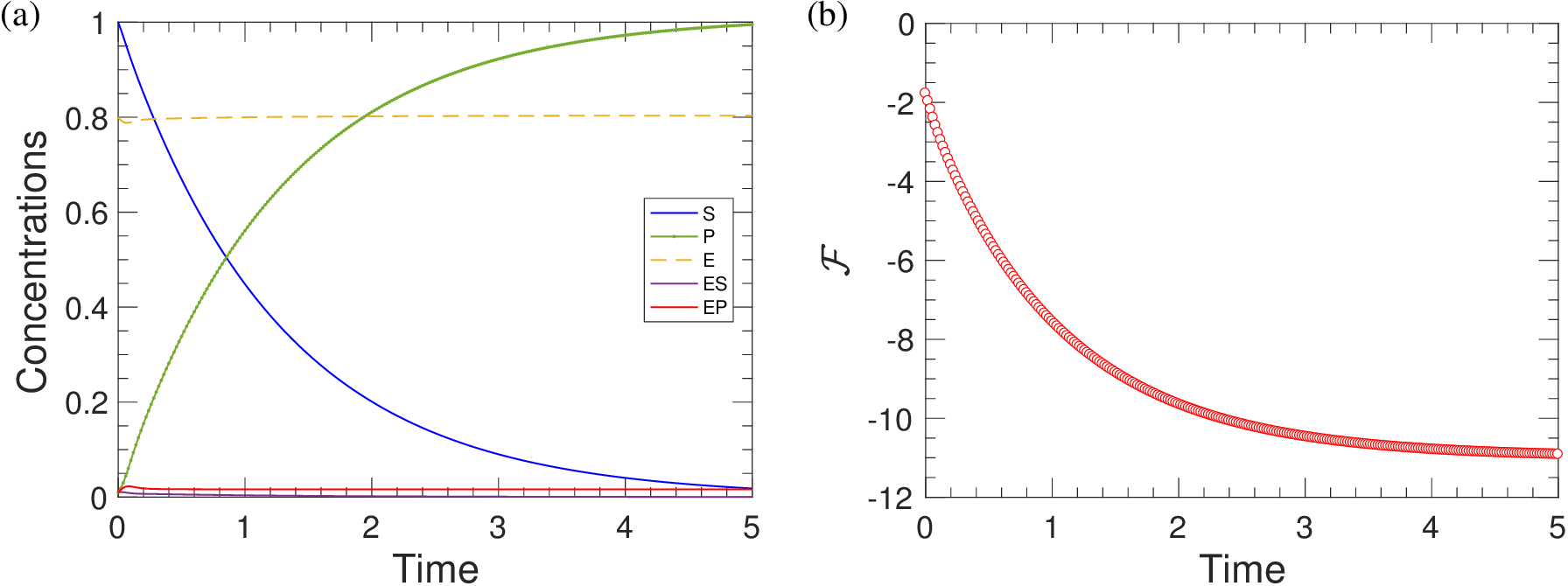}
  \caption{Numerical results for the Generalized Michaelis-Menten kinetics with two intermediate states ($\dt = 1/50$) : (a) The concentrations of different species with respect to time, (b) The numerical free energy with respect to time.}\label{MM_Res}
\end{figure}

\subsection{Reaction-diffusion: accuracy test}

Next we consider a reaction-diffusion system
\begin{equation}\label{example3}
\begin{cases}
  & \pp_t u = D_u \Delta u  - k_1^+ u v^2 + k_1^- v^3 \\
  &  \pp_t v = D_v \Delta v + k_1^+ u v^2 - k_1^- v^3,  \\
\end{cases}
\end{equation}
in which the reaction is given by 
$$ 
  \ce{U + 2 V <=>[k_1^+][k_1^-] 3V} . 
$$ 
Such a type reaction plays an important role in pattern formation in biochemical systems when $k_1^- = 0$ \cite{pearson1993complex, hao2020spatial}. For $k_{1}^{\pm} > 0$, the energy-dissipation law is given by
\begin{equation*}
  \frac{\dd}{\dd t} \int u (\ln u - 1 + U_u) + v (\ln v - 1 + U_v)  \dd \x = -  \int \dot{R} \ln \left( \frac{\dot{R}}{k_1^{-} v^{3}} + 1\right) + \frac{1}{D_u} |\nabla \mu_u|^2 + \frac{1}{D_v} |\nabla \mu_v|^2  \dd \x,
\end{equation*}
where $U_u = \ln k_1^+$ and $U_v = \ln k_1^-$ so that $\dot{R} = k_1^+ u v^2 - k_1^- v^3$.

In the simulation, we take the computational domain $\Omega = (-1, 1)^2$, and impose a periodic boundary condition for both $u$ and $v$ on $\pp \Omega$.
Figure~\ref{RD_1} displays numerical results for $D_u = 0.2$, $D_v = 0.1$, $k_1^+ = 1$ and $k_1^- = 0.1$ at different time instants, with the initial condition given by
\begin{equation*}
u = (- \tanh( (\sqrt{x^2 + y^2} - 0.4)/0.01) + 1)/2 + 1; \quad u = (\tanh( (\sqrt{x^2 + y^2} - 0.4)/0.01) + 1)/2 + 1 . 
\end{equation*}
In the computation, we take $h = 1/50$ and $\Delta t = 1/100$. As shown in Figure~\ref{RD_1}(e), the numerical solution is energy stable for this reaction-diffusion system.

\begin{figure}[!h]
   \includegraphics[width = \linewidth]{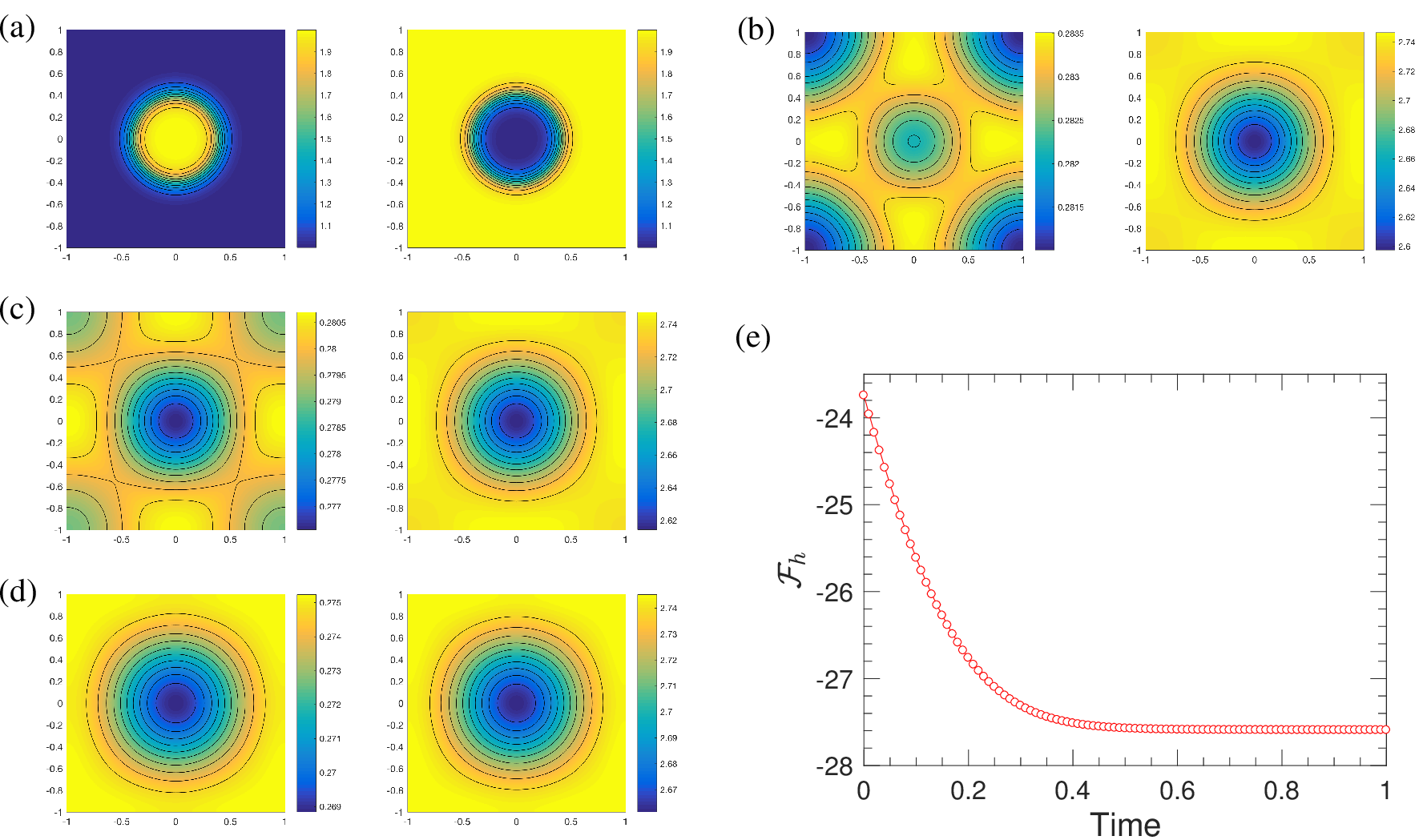}
  \caption{(a)-(d): Numerical results for the reaction-diffusion system (\ref{example3}) with $D_u = 0.2$, $D_v = 0.1$, $k_1^+ = 1$ and $k_1^- = 0.1$ at (a) t = 0, (b) t = 0.7, (c) t = 0.75 and (d) t = 1. (e): The discrete free energy in the time evolution.}\label{RD_1}
\end{figure}

We also perform an accuracy test for this example. The numerical error at $T = 0.2$ is considered, before the system reaches the constant steady state. Since the analytical solution is not available in this example, we use the numerical solution with $h = 1/200$ and $\Delta t = 1/1600$ as the reference solution. 
In the accuracy test for the temporal numerical errors, we fix the spatial resolution as $h=\frac{1}{200}$, so that the spatial numerical error is negligible. Table~\ref{table2} displays the $\ell^{\infty}$ numerical errors at $T = 0.2$, as well as the CPU time for numerical simulations with a sequence of time step sizes: $\dt=\frac{1}{25}$, $\frac{1}{50}$, $\frac{1}{100}$, $\frac{1}{200}$ and $\frac{1}{400}$. 
A very clear first order temporal accuracy for the operator splitting scheme has been observed in the table, as the time step size becomes more refined. 
\begin{table}[!h]
  \begin{center}
    \begin{tabular}{c | c | c |c | c | c | c } 
     \hline \hline
     $\dt$ &  h & $\ell^{\infty}$ error of $u$ & Order in time & $\ell^{\infty}$ error of $v$ & Order in time & CPU (s)  \\
     \hline
     1/25  &  1/200    & 1.117e-1 &    & 9.971e-2 &            &   5.88 \\
     1/50  &  1/200    & 6.045e-2 &  0.8858   & 5.357e-2 &  0.8963    &   11.92 \\
     1/100  &  1/200    &3.083e-2 & 0.9714    & 2.721e-2 &  0.9773          &   23.08 \\
     1/200  & 1/200  & 1.479e-2 & 1.0620     & 1.302e-2 & 1.0634 & 47.08\\
     1/400  & 1/200  & 6.428e-3 & 1.2022 &  5.655e-3 & 1.2031 & 93.86 \\
     \hline  \hline
    \end{tabular}

\end{center}
\caption{Numerical errors, order of accuracy and CPU time for numerical simulations of (\ref{example3}) at $T = 0.2$. The numerical solution with $h = 1/200$ and $\dt = 1/1600$ is taken as the reference solution.}\label{table2}
\end{table}

\begin{table}[ht]
  \begin{center}
  \begin{tabular}{ccccccc}
  \hline  \hline
   --- & $\psi = u$ & Order & $\psi = v$& Order \\
   \hline
  $\| \psi_{h_1} - \psi_{h_2} \|_{\infty}$  & 5.586e-3 & -      & 4.900e-3 & - \\
  $\| \psi_{h_2} - \psi_{h_3} \|_{\infty}$  & 1.979e-3 & 1.970 & 1.727e-3 & 1.983 \\ 
  $\| \psi_{h_3} - \psi_{h_4} \|_{\infty}$  & 9.204e-4 & 1.983 & 8.014e-4 & 1.991 \\
  $\| \psi_{h_3} - \psi_{h_4} \|_{\infty}$  & 5.011e-4 & 1.990 & 4.360e-4 & 1.993 \\
   \hline  \hline
  \end{tabular}
  \caption{The $\ell^\infty$ differences and convergence order for the numerical solutions of $u$ and $v$ at time $T=0.2$. Various mesh resolutions are tested: $h_1=\frac{1}{20}$, $h_2=\frac{1}{30}$, $h_3=\frac{1}{40}$, $h_4=\frac{1}{50}$, $h_5=\frac{1}{60}$, and the time step size is taken as $\Delta t=h^2$.}
  \label{t2:convergence}
  \end{center}
  \end{table}

To verify the spatial accuracy of the proposed numerical scheme for this example, we perform the computations on a sequence of mesh resolutions: $h=\frac{1}{20}, \frac{1}{30}, \frac{1}{40}, \frac{1}{50}$, $\frac{1}{60}$, and the time step size is set as $\Delta t=h^2$ to eliminate the affect of temporal errors.
Since an analytical form of the exact solution is not available, we compute the $\ell^{\infty}$ differences between numerical solutions with consecutive spatial resolutions, $h_{j-1}$, $h_j$ and $h_{j+1}$, in the Cauchy convergence test. 
Since we expect the numerical scheme preserves a second order spatial accuracy, we can compute the following quantity
$$
   \frac{ \ln \Big(  \frac{1}{A^*} \cdot 
   \frac{\| u_{h_{j-1}} - u_{h_j} \|_\infty }{ \| u_{h_j} - u_{h_{j+1}} \|_\infty} \Big) } 
   {\ln  \frac{h_{j-1}}{h_j} } ,  \quad A^* =  \frac{ 1 - \frac{h_j^2}{h_{j-1}^2} }{1 - \frac{h_{j+1}^2}{h_j^2} } ,  
   \quad \mbox{for} \, \, \, h_{j-1} > h_j > h_{j+1} , 
$$
to check the convergence order \cite{liu2020positivity}. As demonstrated in Table~\ref{t2:convergence}, the numerical errors at time $T=0.2$ improve robustly as the spatial mesh refines, and an almost perfect second order spatial convergence rate for the proposed operator splitting scheme is observed. 

\subsection{Reaction-diffusion: nonlinear diffusion}

\begin{figure}[!h]
  \includegraphics[width = \linewidth]{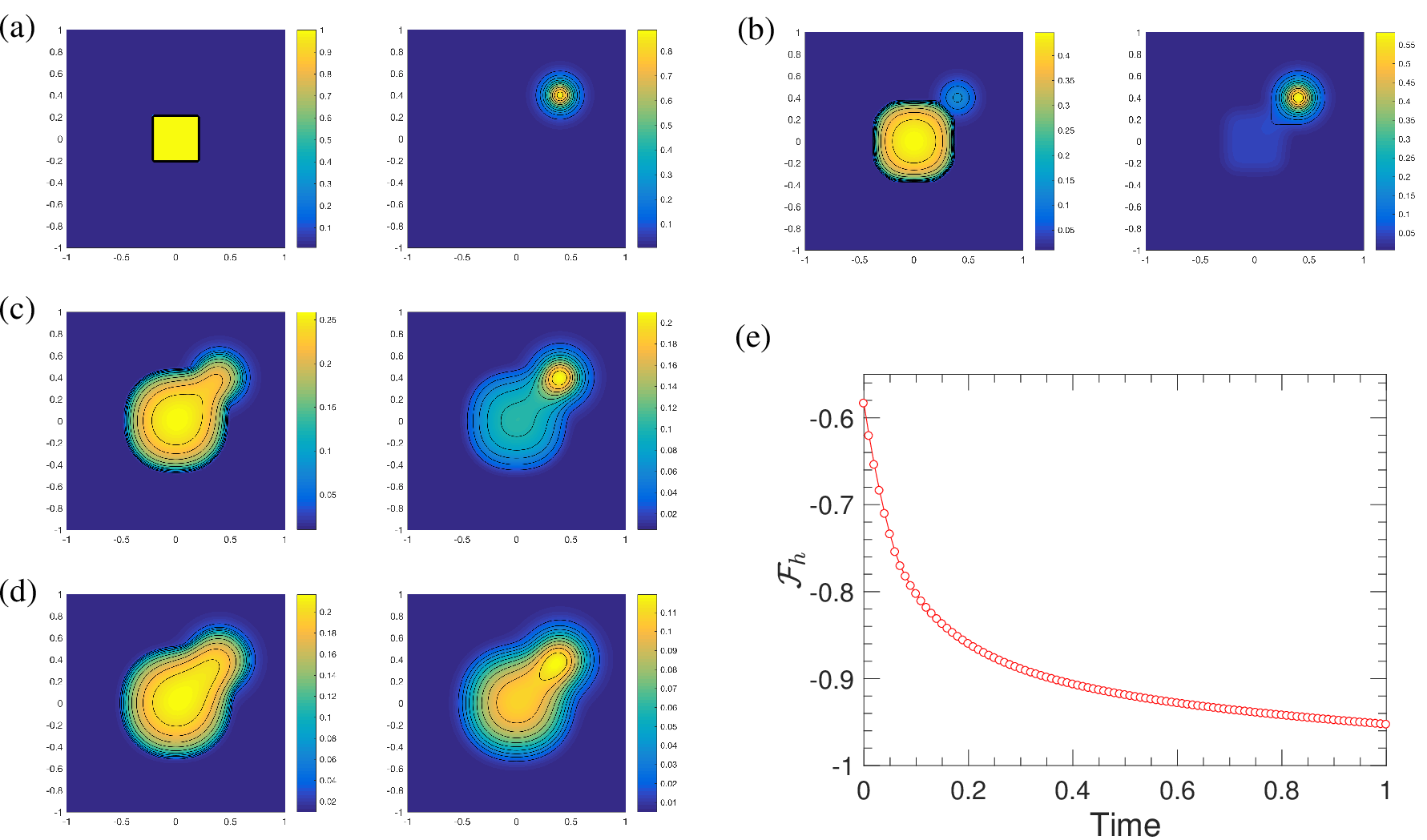}
  \caption{(a)-(d): Numerical results for the reaction-diffusion system (\ref{example_4}) with $m = 4, \quad k_1 = 2, \quad k_2 = 1, \quad \epsilon  = 0.01$ at (a) t = 0, (b) t = 0.7, (c) t = 0.75 and (d) t = 1. (e): The discrete free energy in the time evolution.}\label{RD_PME}
\end{figure}

In the last example, we consider a case that couples a chemical reaction with a nonlinear diffusion given by
\begin{equation}\label{example_4}
  \begin{cases}
  & a_t = \Delta a^{m} - k_1 a + k_2 b \\
  &  b_t  = \epsilon \Delta a + k_1 a - k_2 b. \\
  \end{cases}
\end{equation}
The corresponding chemical reaction is given by $\ce{A <=>[k_1][k_2] B}$.
Equation \ref{example_4} has potential applications in modeling the tumor growth (see \cite{liu2018accurate, perthame2014hele} for more discussions on modeling perspectives). Here we only use equation (\ref{example_4}) as the numerical test; the values that we take may not have the physical meaning. The corresponding energy-dissipation law for (\ref{example_4}) can be written as
\begin{equation}
\frac{\dd}{\dd t} \int a (\ln a - 1) + b (\ln b - 1) + a U_a + b U_b \dd \x = - \int \dot{R} \ln \left( \frac{\dot{R}}{k_2 b} + 1 \right) + \frac{1}{m} a^{2-m} |\uvec_a|^2 +   \frac{1}{\epsilon} b |\uvec_b|^2 \dd \x.
\end{equation}
We take $U_a = \ln k_1$ and $U_b = \ln k_2$ to obtain the original equation. In the numerical simulation, we set the domain as $\Omega = (-1, 1)^2$ and impose a periodic boundary conditions for both $a$ and $b$. The parameters in the equations are chosen as
\begin{equation}
  m = 4, \quad k_1 = 2, \quad k_2 = 1, \quad \epsilon  = 0.01.
\end{equation}
Figure \ref{RD_PME} shows the numerical results for $h = 1/50$ and $t = 1/100$, in which the initial condition is given by 
\begin{equation*}
a(x, y) = 
\begin{cases}
  & 1 , \quad   |x| \leq 0.2 ~\text{and}~ |y| \leq 0.2, \\
  & 0.01, \quad \text{otherwise}, \\
\end{cases}
\quad
b(x, y) = - \tanh( (\sqrt{(x - 0.4)^2 + (y - 0.4)^2} - 0.1)/0.1) + 0.005.
\end{equation*}
It is noticed that the chemical reaction is at its chemical equilibrium for $a = 0.01$ and $c = 0.005$. This numerical example has also demonstrated the energy stability and positivity-preserving property of the proposed operator splitting scheme. In addition, the finite speed propagation property for the PME can be observed in the numerical solutions.

\section{Summary}

We propose a positivity-preserving and energy stable operator splitting numerical scheme 
for reaction-diffusion systems involving the law of mass action with detailed balance condition. The numerical discretization is based on a recently-developed energetic variational formulation, in which the reaction part is reformulated as a generalized gradient flow of the reaction trajectory and both the reaction and diffusion parts dissipates the same free energy.
This discovery enables us to develop an operator splitting scheme, with both the positivity-preserving property and energy stability established at a theoretical level. In more details, the convex splitting technique is applied to the reformulated reaction part; all the logarithmic terms are treated implicitly because of their convex nature. The positivity-preserving property and unique solvability are theoretically proved, based on the subtle fact that, the singularity of the logarithmic function around the limiting value  prevents the numerical solution reaching these singular values. The energy stability follows a careful convexity analysis. Using similar ideas, we establish the positivity-preserving property and energy stability for the standard semi-implicit solver in the diffusion stage as well. Therefore, a combination of these two stages leads to a numerical scheme satisfying both theoretical properties. Several numerical results have also been presented to demonstrate the robustness of the proposed operator splitting scheme. To our best knowledge, it is the first time to report a free energy dissipation for a combined operator splitting scheme to a nonlinear PDE with variational structures. This idea can be applied to a general dissipative systems that contain different components of dissipations, such problems with dynamical boundary conditions \cite{knopf2020phase, liu2019energetic}, and chemomechanical systems in biology and soft matter physics \cite{julicher1997modeling, wang2021two}.

\section*{Acknowledgement} 
This work is partially supported by the National Science Foundation (USA) grants NSF DMS-1759536, NSF DMS-1950868 (C. Liu, Y. Wang) and NSF DMS-2012669 (C. Wang). 
Y. Wang would also like to thank Department of Applied Mathematics at Illinois Institute of Technology for their generous support and for a stimulating environment.


\section*{References}
\bibliographystyle{cas-model2-names}
\bibliography{KCR}

\end{document}